\documentclass[11pt]{amsart}
\usepackage{latexsym,amscd,amssymb, graphicx, color, amsthm, bm, amsmath, cancel, enumitem,young,chessboard}  
\usepackage{tikz}
\usepackage[margin=1in]{geometry}
\usepackage{hyperref}
\usepackage[all,cmtip]{xy}

\numberwithin{equation}{section}

\newtheorem{theorem}{Theorem}[section]
\newtheorem{proposition}[theorem]{Proposition}
\newtheorem{corollary}[theorem]{Corollary}
\newtheorem{lemma}[theorem]{Lemma}

\newtheorem{problem}[theorem]{Problem}
\newtheorem{example}[theorem]{Example}
\newtheorem{remark}[theorem]{Remark}

\theoremstyle{definition}
\newtheorem{defn}[theorem]{Definition}

\newcommand{\Hilb}{{\mathrm{Hilb}}}

\newcommand{\Conf}{{\mathrm{Conf}}}

\newcommand{\Ind}{{\mathrm{Ind}}}

\newcommand{\ch}{{\mathrm{ch}}}

\newcommand{\VG}{{\mathrm{VG}}}

\newcommand{\codim}{{\mathrm{codim}}}

\newcommand{\symm}{{\mathfrak{S}}}
\newcommand{\II}{{\mathbf{I}}}
\newcommand{\RRR}{{\mathbf{R}}}
\newcommand{\zero}{{\mathbf{0}}}
\newcommand{\gr}{{\mathrm {gr}}}
\newcommand{\BBB}{{\mathcal{B}}}
\newcommand{\TTT}{{\mathcal{T}}}

\newcommand{\NNN}{{\mathcal{N}}}

\newcommand{\AAA}{{\mathcal{A}}}
\newcommand{\JJJ}{{\mathcal{J}}}

\newcommand{\Supp}{{\mathrm{Supp}}}
\newcommand{\Flat}{{\mathrm{Flat}}}

\newcommand{\CCC}{{\mathcal{C}}}

\newcommand{\KKK}{{\mathcal{K}}}
\newcommand{\CC}{{\mathbb{C}}}
\newcommand{\QQ}{{\mathbb{Q}}}
\newcommand{\ZZ}{{\mathbb{Z}}}
\newcommand{\PP}{{\mathbb{P}}}
\newcommand{\FF}{{\mathbb{F}}}
\newcommand{\RR}{{\mathbb{R}}}

\newcommand{\VVV}{{\mathcal{VG}}}

\newcommand{\LLL}{\mathcal{L}}

\newcommand{\FFF}{\mathcal{F}}

\newcommand{\MMM}{\mathcal{M}}
\newcommand{\III}{\mathcal{I}}
\newcommand{\Sep}{\mathrm{Sep}}

\newcommand{\Ker}{\mathrm{Ker}}

\newcommand{\yyy}{{\mathbf{y}}}

\newcommand{\zzz}{{\mathbf{z}}}

\newcommand{\Zpoints}{\mathcal{Z}}
\newcommand{\Ypoints}{\mathcal{Y}}

\newcommand{\cyc}{\text{cyc}}

\newcommand{\Aut}{\mathrm{Aut}}

\begin{document}

\title[Big Varchenko--Gelfand rings and Orbit Harmonics]
{Big Varchenko--Gelfand rings and Orbit Harmonics}

\author{Brendon Rhoades}

\address{Department of Mathematics, UC San Diego, La Jolla, CA, 92039, USA}
\email{bprhoades@ucsd.edu}

\begin{abstract}
    Let $\MMM$ be a conditional oriented matroid. We define a graded algebra $\widehat{\VVV}_\MMM$ with vector space dimension given by the number of covectors in $\MMM$ which admits a distinguished filtration indexed by the poset $\LLL(\MMM)$ of flats of $\MMM$. The subquotients of this filtration are isomorphic to graded Varchenko--Gelfand rings of contractions of $\MMM$, so we call $\widehat{\VVV}_\MMM$ the {\em graded big Varchenko--Gelfand ring of $\MMM$.} We describe a no broken circuit type basis of $\widehat{\VVV}_\MMM$ and study its equivariant structure under the action of $\Aut(\MMM)$. Our key technique is the orbit harmonics deformation which encodes $\widehat{\VVV}_\MMM$ (as well as the classical Varchenko--Gelfand ring) in terms of a locus of points. 
\end{abstract}

\maketitle

\section{Introduction}
\label{sec:Introduction}

Fix a field $\FF$. In this paper we define and study a new graded $\FF$-algebra attached to a real hyperplane arrangement $\AAA$ or, more generally, a conditional oriented matroid $\MMM$. The vector space dimension of this algebra counts faces of $\AAA$ (or covectors of $\MMM$). Our new algebra is inspired by more established algebras whose dimensions count chambers of $\AAA$ (or topes of $\MMM$). We recall the classical story and outline our new construction.

\subsection{Chambers and topes}
Let $V$ be a finite-dimensional real vector space, let $\AAA = \{H_i \,:\, i \in \III\}$ be a finite affine hyperplane arrangement in $V$, and let $\CCC$ be the set of chambers of $\AAA$. The {\em Varchenko--Gelfand ring}  $\VG_\AAA$ (with  coefficients in $\FF$) is the set of functions $\CCC \to \FF$ endowed with pointwise operations. For $i \in \III$ the Heaviside functions $h_i^\pm \in \VG_\AAA$ are the indicator functions for whether a given chamber $C \in \CCC$ lies on the positive or negative side of $H_i$. These functions generate $\VG_\AAA$ as an algebra. The {\em graded Varchenko-Gelfand ring} is the associated graded ring $\VVV_\AAA$ arising from this filtration.

The  ring $\VVV_\AAA$ is a commutative (rather than anticommutative) version of the more widely known {\em Orlik--Solomon algebra} \cite{OS} of $\AAA$. Varchenko and Gelfand \cite{VG} gave a presentation for $\VVV_\AAA$ as a polynomial ring quotient mirroring the exterior algebra quotient presentation of the Orlik--Solomon algebra. Moseley gave \cite{Moseley} a topological interpretation of $\VVV_\AAA$ as a cohomology ring where the complexified complements of the Orlik--Solomon setting \cite{Brieskorn, OS} are replaced by a complement of $\AAA$ within the triplication $V \oplus V \oplus V$ of the real vector space $V$.

There are many variations on $\VG_\AAA$ and $\VVV_\AAA$. Gelfand and Rybnikov \cite{GR} introduced a version of these rings for oriented matroids. Dorpalen-Barry \cite{DB} studied a variant which depends on an open cone $\KKK \subseteq V$. Dorpalen-Barry, Proudfoot, and Wang \cite{DPW} defined the most general version of these rings for {\em conditional oriented matroids}, a relatively new object in discrete geometry introduced by Bandelt, Chepoi, and Knauer \cite{BCK}; we recall their definition.

Let $\III$ be a finite ground set. A {\em signed subset} of $\III$ is a function $X : \III \to \{+,-,0\}$.\footnote{Other authors encode signed subsets pairs $X = (X_+,X_-)$ of disjoint subsets of $\III$.} 
Given two signed subsets $X$ and $Y$ of $\III$, their {\em separating set} is
\begin{equation}
    \Sep(X,Y) := \{ i \in \III \,:\, X(i) = -Y(i) \neq 0 \}
\end{equation}
and their {\em composition} $X \circ Y: \III \to \{+,-,0\}$ is given by
\begin{equation}
    (X \circ Y)(i) := \begin{cases}
        X(i) & \text{if $X(i) \neq 0$,} \\
        Y(i) & \text{otherwise.}
    \end{cases}
\end{equation}

A {\em conditional oriented matroid} on $\III$ is a family $\MMM$ of signed subsets of $\III$ which satisfies the following axioms.\footnote{The symbol $\LLL$ is more commonly used than $\MMM$ to refer to conditional oriented matroids. We reserve the use of $\LLL$ for posets of flats.}
\begin{enumerate}
    \item (Face Symmetry) If $X, Y \in \MMM$ then $X \circ -Y \in \MMM$.
    \item (Strong Elimination) If $X, Y \in \MMM$ and $i \in \Sep(X,Y)$, there exists $Z \in \MMM$ so that $Z(i) = 0$ and $Z(j) = (X \circ Y)(j)$ for all $j \in \III - \Sep(X,Y)$.
\end{enumerate}
Signed sets $X \in \MMM$ are called {\em covectors}. Since 
\[
X \circ Y = (X \circ -X) \circ Y = X \circ (-X \circ Y) = X \circ -(X \circ -Y)
\]
it follows from these axioms that $\MMM$ is closed under composition. A covector $X \in \MMM$ is a {\em tope} if $X(i) \neq 0$ for all $i \in \III$. We write $\TTT(\MMM)$ for the collection of topes of $\MMM$. Conditional oriented matroids relate to arrangements as follows.

\begin{example}
    \label{ex:com-hyperplane}
    Let $V$ be a finite-dimensional real vector space, let $\KKK \subseteq V$ be an open convex subset of $V$, and let $\{\alpha_i : V \to \RR \,:\, i \in \III \}$ be a finite collection of affine linear forms which defines an  arrangement $\AAA$ with hyperplanes $\alpha_i^{-1}(0)$. For any relatively open face $F$ of $\AAA$, one has a signed set $X_F : \III \to \{+,-,0\}$ given by
    \begin{equation}
        X_F(i) := \begin{cases}
            + & \text{if $\alpha_i(F) > 0$,} \\
            - & \text{if $\alpha_i(F) < 0$,} \\
            0 & \text{if $\alpha_i(F) = 0$.} 
        \end{cases}
    \end{equation}
    The family $\{ X_F \,:\, F \cap \KKK \neq \varnothing \}$ is a conditional oriented matroid on $\III$. In this context, the composition operation $\circ$ is the {\em face semigroup} operation.
\end{example}

In Figure~\ref{fig:four-lines} we have $V = \RR^2$, $\KKK$ is an open ellipital region, and $\AAA = \{H_1,H_2,H_3,H_4\}$ consists of four lines whose positive sides are indicated. The arrangement $\AAA$ partitions $\KKK$ into 13 faces, one of which is a bolded line segment with covector $(0,+,-,+)$. We will refer to Figure~\ref{fig:four-lines} throughout the paper.

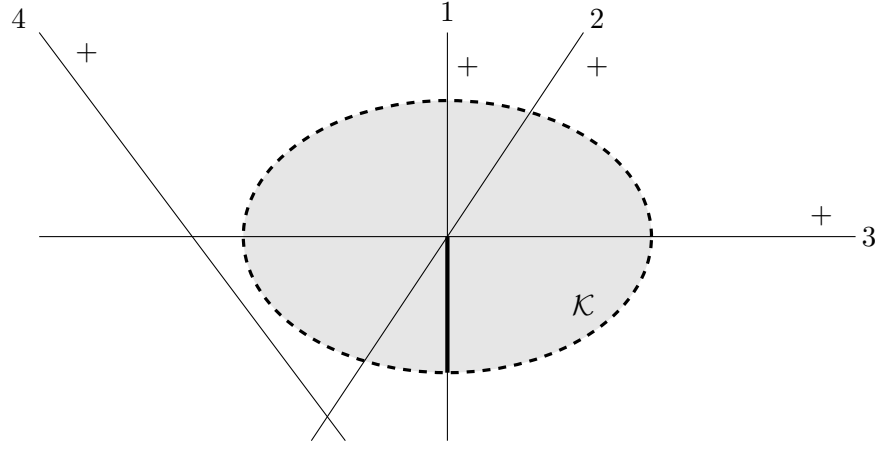
\begin{figure}
\begin{center}
    \begin{tikzpicture}[scale = 0.9]
        \draw[fill=gray!20, very thick, dashed] (0,0) ellipse (3cm and 2cm);

        \node at  (2,-1) {$\KKK$};

        \draw[-] (0,3) -- (0,-3);
        \node at (0.3,2.5) {$+$};
        \node at (0,3.3) {$1$};

        \draw[-] (2,3) -- (-2,-3);
        \node at (2.2,2.5) {$+$};
        \node at (2.2,3.2) {$2$};

        \draw [-] (-6,0) -- (6,0);
        \node at (5.5,0.3) {$+$};
        \node at (6.2,0) {$3$};

        \draw [-] (-6,3) -- (-1.5,-3);
        \node at (-5.3,2.7) {$+$};
        \node at (-6.3,3.2) {$4$};

        \draw [-, ultra thick] (0,0) -- (0,-2);
 
    \end{tikzpicture}
\end{center}
\caption{An arrangement of four lines.}
\label{fig:four-lines}
\end{figure}

Dorpalen-Barry, Proudfoot, and Wang studied \cite{DPW} the algebra $\VG_\MMM$ of functions on the set $\TTT(\MMM)$ of topes of $\MMM$. As in the case of arrangements, the algebra $\VG_\MMM$ admits a natural filtration via Heaviside functions which gives rise to an associated graded ring $\VVV_\MMM$. Dorpalen-Barry et. al. established \cite[Thm. 1.6]{DPW} a presentation of $\VG_\MMM$ and $\VVV_\MMM$ as polynomial ring quotients. A topological interpretation of these rings is also proven \cite[Thm. 1.1]{DPW} when $\MMM$ is as in Example~\ref{ex:com-hyperplane}.

\subsection{Faces and covectors}
An affine arrangement $\AAA$ in a real vector space $V$ partitions the space $V$ into {\em faces} of various dimensions; the chambers are the faces of maximal dimension. Varchenko and Gelfand introduced \cite{VG} the algebra $\widehat{\VG}_\AAA$ of functions $\FFF \to \FF$ where $\FFF$ is the set of faces of $\AAA$. We call $\widehat{\VG}_\AAA$ the {\em big Varchenko-Gelfand ring} of $\AAA$. Varchenko and Gelfand gave a presentation \cite{VG} of $\widehat{\VG}_\AAA$ in terms of Heaviside functions; Gelfand and Rybnikov \cite{GR} did the same in the context of oriented matroids.

As with the `small' VG ring, the Heaviside generators of $\widehat{\VG}_\AAA$ give  a filtration on $\widehat{\VG}_\AAA$. We write $\widehat{\VVV}_\AAA$ for the associated graded ring of this filtration and call $\widehat{\VVV}_\AAA$ the {\em graded big Varchenko-Gelfand ring} of $\AAA$. The ring $\widehat{\VVV}_\AAA$ appears to have not been studied before this paper. We prove the following results about this new graded ring. If $X$ is a flat of $\AAA$, we write $\AAA^X$ for the restriction of $\AAA$ to $X$. Recall that the {\em Hilbert series} of a graded vector space $W = \bigoplus_{d \geq 0} W_d$ is the formal power series $\Hilb(W;q) := \sum_{d \geq 0} \dim(W_d) \cdot q^d$.
\begin{itemize}
    \item We give an explicit presentation (see Theorem~\ref{thm:big-locus-identification}) of $\widehat{\VVV}_\AAA$ as a polynomial ring quotient.
    \item The big graded VG ring $\widehat{\VVV}_{\AAA}$ admits a filtration indexed by the poset of flats $\LLL(\AAA)$. The subquotient indexed by a flat $X \in \LLL(\AAA)$ is the `small' graded VG ring $\VVV_{\AAA^X}$ with degree shifted up by $\codim(X)$. (See Theorem~\ref{thm:big-locus-identification}.)
    \item The Hilbert series of $\widehat{\VVV}_\AAA$ is given (see Corollary~\ref{cor:big-hilbert-series}) by
    \[
    \Hilb(\widehat{\VVV}_\AAA;q) = \sum_{X \in \LLL(\AAA)} q^{\codim(X)} \cdot \Hilb(\VVV_{\AAA^X};q).
    \]
    \item If $\AAA$ is invariant under the action of a finite linear group $G \subseteq GL(V)$ where $\# G$ is nonzero in $\FF$, one has an isomorphism (see Theorem~\ref{thm:graded-module-structure}) of graded $G$-modules
    \[
    \widehat{\VVV}_\AAA \cong \bigoplus_{[X] \in \LLL(\AAA)/G} \Ind_{G_X}^{G}(\VVV_{\AAA^X})(-\codim(X))
    \]
    where $G_X = \{g \in G \,:\, g \cdot X = X \}$ and $(-\codim(X))$ shifts degree up by $\codim(X)$.
\end{itemize}
The results roughly state that the big VG ring for $\AAA$ is built out of small VG rings for the restrictions $\AAA^X$ for $X \in \LLL(\AAA)$. This is an algebraic enrichment of the fact that any face of $\AAA$ is a chamber of a unique restriction $\AAA^X$. These results are no more difficult to state and prove in the broader context of conditional oriented matroids; Theorem~\ref{thm:big-locus-identification}, Corollary~\ref{cor:big-hilbert-series}, and Theorem~\ref{thm:graded-module-structure} are stated in this language.

\subsection{Orbit harmonics}
The {\em orbit harmonics} technique of deformation theory stands behind the results in this paper. Let $\Zpoints \subseteq \FF^n$ be a finite point set and let $S = \FF[x_1,\dots,x_n]$ be the coordinate ring of $\FF^n$. We have the vanishing ideal $\II(\Zpoints) \subseteq S$ given by
\begin{equation}
    \II(\Zpoints) := \{ f \in S \,:\, f(\zzz) = 0 \text{ for all } \zzz \in \Zpoints \}.
\end{equation}
Since $\Zpoints$ is finite, Lagrange interpolation gives the isomorphism
\begin{equation}
\label{eq:orbit-harmonics-preliminary}
    \FF[\Zpoints] \cong S/\II(\Zpoints)
\end{equation}
of ungraded $\FF$-vector spaces of dimension $\# \Zpoints$. The {\em associated graded ideal} $\gr \, \II(\Zpoints)$ is the ideal
\begin{equation}
    \gr \, \II(\Zpoints) := (\tau(f) \,:\, f \in \II(\Zpoints), \, f \neq 0 ) \subseteq S
\end{equation}
where $\tau(f)$ is the top-degree homogeneous component of a nonzero polynomial $f$. Explicitly, if $f = f_d + \cdots + f_1 + f_0$ for $f_i$ homogeneous of degree $i$ and $f_d \neq 0$, we have $\tau(f) = f_d$. The ideal $\gr \, \II(\Zpoints)$ is homogeneous by construction and the vector space isomorphism \eqref{eq:orbit-harmonics-preliminary} extends to an isomorphism
\begin{equation}
\label{eq:orbit-harmonics-true}
    \FF[\Zpoints] \cong S/\II(\Zpoints) \cong S/\gr \, \II(\Zpoints) =: \RRR(\Zpoints)
\end{equation}
of vector spaces. The quotient $\RRR(\Zpoints) = S/\gr \, \II(\Zpoints)$ has the additional structure of a graded vector space. If the locus $\Zpoints$ is stable under the action of a finite matrix group $G \subseteq GL_n(\FF)$ and $\# G \in \FF^{\times}$, we may interpret \eqref{eq:orbit-harmonics-true} as an isomorphism of $G$-modules where $S/\gr \, \II(\Zpoints)$ has the additional structure of a graded $G$-module. 

As illustrated in the diagram below, the geometric interpretation of the orbit harmonics deformation $S/\II(\Zpoints) \leadsto S/\gr \, \II(\Zpoints)$  is the flat limit of the linear deformation of the reduced locus $\Zpoints$ to a fatpoint of degree $\# \Zpoints$ supported at the origin. The orbit harmonics deformation was introduced by Kostant \cite{Kostant}; in his context the point set $\Zpoints$ was a regular orbit of the action of a complex reflection group $W$ on its reflection representation and $\RRR(\Zpoints)$ was the $W$-coinvariant ring. Orbit harmonics has appeared in numerous places since then, with strategically chosen loci $\Zpoints$ giving rise to rings $\RRR(\Zpoints)$ related to Springer fibers \cite{GP}, Macondald-theoretic delta operators \cite{HRS}, and the Viennot shadow line avatar of the Schensted correspondence \cite{Rhoades} among other things.

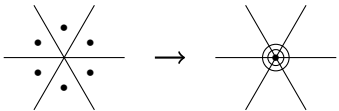
\begin{center}
 \begin{tikzpicture}[scale = 0.2]
\draw (-4,0) -- (4,0);
\draw (-2,-3.46) -- (2,3.46);
\draw (-2,3.46) -- (2,-3.46);

 \fontsize{5pt}{5pt} \selectfont
\node at (0,2) {$\bullet$};
\node at (0,-2) {$\bullet$};

\node at (-1.73,1) {$\bullet$};
\node at (-1.73,-1) {$\bullet$};
\node at (1.73,-1) {$\bullet$};
\node at (1.73,1) {$\bullet$};

\draw[thick, ->] (6,0) -- (8,0);

\draw (10,0) -- (18,0);
\draw (12,-3.46) -- (16,3.46);
\draw (12,3.46) -- (16,-3.46);

\draw (14,0) circle (15pt);
\draw(14,0) circle (25pt);
\node at (14,0) {$\bullet$};

 \end{tikzpicture}
\end{center}

Let $\MMM$ be a conditional oriented matroid on the ground set $\III$. Consider the vector space $\FF^{\III \times \{+,-,0\}}$ with basis given by pairs
$(i,s)$ with $i \in \III$ and $s \in \{+,-,0\}$.
A covector $X \in \MMM$ gives rise to a point $\zzz_X \in \FF^{\III \times \{+,-,0\}}$ with $0,1$-coordinates in a natural way. Let $\Zpoints_\MMM = \{ \zzz_X \,:\, X \in \MMM \}$ be the locus of all such points; see Definition~\ref{def:big-locus} for details. We have the identification (see Proposition~\ref{prop:big-orbit-harmonics-interpretations})
\begin{equation}
\label{eq:intro-big-graded}
    \widehat{\VVV}_\MMM \cong \RRR(\Zpoints_\MMM)
\end{equation}
of graded algebras where $\widehat{\VVV}_\MMM$ is the big graded VG ring associated to $\MMM$. Equation~\eqref{eq:intro-big-graded} is a  succinct way to define $\widehat{\VVV}_\MMM$, but it would be useful to have explicit generators of the defining ideal $\gr \, \II(\Zpoints_\MMM)$ of $\RRR(\Zpoints_\MMM)$.  In Theorem~\ref{thm:big-locus-identification} we find such generators.

Due to its connection to big VG rings, we call $\Zpoints_\MMM$ the {\em big locus}. There is also a natural `small locus' $\Ypoints_\MMM$ with points indexed by the topes $\TTT(\MMM)$ in $\MMM$. In Section~\ref{sec:Small} we reformulate the results of Dorpalen-Barry, Proudfoot, and Wang \cite{DPW} on the small VG rings $\VG_\MMM$ and $\VVV_\MMM$ in terms of the small locus $\Ypoints_\MMM$.

\subsection{Organization} The rest of the paper is organized as follows. {\bf Section~\ref{sec:Discrete}} recalls various notions from (conditional, oriented) matroid theory and makes some easy combinatorial observations which will be used later on. The short {\bf Section~\ref{sec:Small}} recasts the work of Dorpalen-Barry, Proudfoot, and Wang via the orbit harmonics of the small locus $\Ypoints_\MMM$. {\bf Section~\ref{sec:Big}} is the heart of the paper and contains our main results on the big locus $\Zpoints_\MMM$ and the big graded VG ring $\widehat{\VVV}_\MMM$. {\bf Section~\ref{sec:Braid}} describes our results in the special case of the type A braid arrangement and poses a topological open question (Problem~\ref{prob:topological-interpretation}).

\section*{Acknowledgements}

The author is very grateful to Vic Reiner for ideas which initiated this project and inspiring conversations throughout. The author thanks Ethan Partida and Colin Crowley for pointing out the relationship with matroid Schubert varieties of Remark~\ref{rmk:HW}. The author thanks Jaeseong Oh for pointing out the connection between Equation~\eqref{eq:big-hilbert-series-braid} and necklaces.
The author was partially supported by NSF Grant DMS-2246846.

\section{Discrete Geometry}
\label{sec:Discrete}

This section collects basic results about (conditional  oriented) matroids which will be used later. We refer the reader to e.g. \cite{Oxley} for a comprehensive introduction to matroid theory.

\subsection{Matroids} Let $E$ be a finite ground set. A {\em matroid} $M$ on $E$ consists of a nonempty collection $\BBB$ of subsets of $E$ (called {\em bases}) such that the following exchange condition holds:
\begin{quote}
    if $B_1 \neq B_2$ are bases and $b_1 \in B_1 - B_2$, there exists $b_2 \in B_2 - B_1$ such that the set $(B_1 - \{b_1\}) \cup \{b_2\}$ is a basis.
\end{quote}
Every basis of $M$ has the same cardinality. A subset $I \subseteq E$ is {\em independent} if $I \subseteq B$ for some basis $B$. The {\em rank} of an arbitrary subset $S \subseteq E$ is 
\begin{equation}
    r(S) := \max \{ \# I \,:\, I \subseteq S \text{ is independent} \}.
\end{equation}
For a fixed integer $k \geq 0$, a {\em flat} a maximal subset $F \subseteq E$ of rank $k$. The collection $\LLL(M)$ of all flats of $M$ is called the {\em lattice of flats}; it is a geometric lattice when partially ordered by inclusion. If $F \in \LLL(M)$ is a flat, a {\em basis of $F$} is a maximal independent subset $B_F \subseteq F$. The following is well known. 

\begin{lemma}
    \label{lem:flat-containment-bases}
    Suppose $M$ is a matroid and $F_1 \subseteq F_2$ are flats of $M$. There exist bases $B_{F_i}$ of $F_i$ for $i=1,2$ so that $B_{F_1} \subseteq B_{F_2}$.
\end{lemma}

\subsection{Flat poset of a conditional oriented matroid}
The author was unable to find an explicit reference to posets of flats for conditional oriented matroids; we set up the relevant terminology. Let $\MMM$ be a conditional oriented matroid on the ground set $\III$. If $X \in \MMM$ is a covector, the {\em flat} of $X$ is the subset 
\begin{equation}
    \Flat(X) := \{ i \in \III \,:\, X(i) = 0 \}.
\end{equation}
The {\em poset of flats} is the collection
\begin{equation}
    \LLL(\MMM)  := \{ \Flat(X) \,:\, X \in \MMM \}
\end{equation}
of subsets of $\III$, partially ordered by containment.\footnote{There is an unfortunate notational reversal between the settings of hyperplane arrangements and conditional oriented matroids. For an arrangement $\AAA$, the letter $X$ usually denotes a flat while the letter $F$ usually denotes a face. For a conditional oriented matroid, the letter $X$ is reserved for covectors (i.e. faces) and we use $F$ to denote flats.}

If $\MMM$ is the conditional oriented matroid shown in Figure~\ref{fig:four-lines}, the poset of flats is as follows.

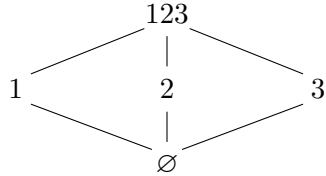
\begin{center}
    \begin{tikzpicture}[scale = 1]

    \node at (0,0) {$\varnothing$};

    \node at (-2,1) {$1$};
    \node at (0,1) {$2$};
    \node at (2,1) {$3$};

    \node at (0,2) {$123$};

    \draw [-] (0,0.3) -- (0,0.7);
    \draw [-] (0,1.3) -- (0,1.7);
    \draw [-] (0.2,0.2) -- (1.8,0.8);
    \draw [-] (0.3,1.8) -- (1.8,1.2);
    \draw [-] (-0.2,0.2) -- (-1.8,0.8);
    \draw [-] (-0.3,1.8) -- (-1.8,1.2);
    \end{tikzpicture}
\end{center}
If we alter the convex open set $\KKK$ as follows
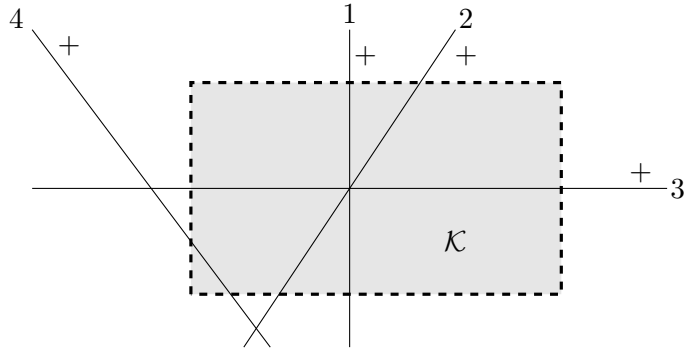
\begin{center}
    \begin{tikzpicture}[scale = 0.7]
        \draw[fill=gray!20, very thick, dashed] (-3,-2) rectangle (4,2);

        \node at  (2,-1) {$\KKK$};

        \draw[-] (0,3) -- (0,-3);
        \node at (0.3,2.5) {$+$};
        \node at (0,3.3) {$1$};

        \draw[-] (2,3) -- (-2,-3);
        \node at (2.2,2.5) {$+$};
        \node at (2.2,3.2) {$2$};

        \draw [-] (-6,0) -- (6,0);
        \node at (5.5,0.3) {$+$};
        \node at (6.2,0) {$3$};

        \draw [-] (-6,3) -- (-1.5,-3);
        \node at (-5.3,2.7) {$+$};
        \node at (-6.3,3.2) {$4$};
 
    \end{tikzpicture}
\end{center}
the new poset of flats is shown below.
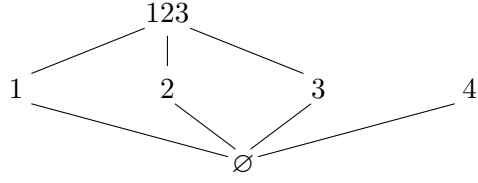
\begin{center}
    \begin{tikzpicture}[scale = 1]

    \node at (0,0) {$\varnothing$};

    \node at (-3,1) {$1$};
    \node at (-1,1) {$2$};
    \node at (1,1) {$3$};
    \node at (3,1) {$4$};

    \node at (-1,2) {$123$};

    \draw [-] (-0.2,0.1) -- (-2.8,0.8);
    \draw [-] (-1,1.3) -- (-1,1.7);
    \draw [-] (-0.1,0.2) -- (-0.9,0.8);
    \draw [-] (-0.7,1.8) -- (0.8,1.2);
    \draw [-] (0.1,0.2) -- (0.9,0.8);
    \draw [-] (-1.3,1.8) -- (-2.8,1.2);
    \draw [-] (0.2,0.1) -- (2.8,0.8);
    \end{tikzpicture}
\end{center}
This example shows that the poset $\LLL(\MMM)$ need not have a maximum element and that maximal chains in $\LLL(\MMM)$ can have different lengths. However, we have the following lemma.

\begin{lemma}
    \label{lem:flat-semilattice}
    Let $\MMM$ be a conditional oriented matroid. If $F_1$ and $F_2$ are flats of $\MMM$, so is $F_1 \cap F_2$. In particular, the poset $\LLL(\MMM)$ is a meet-semilattice.
\end{lemma}

\begin{proof}
    If $X_1, X_2 \in \MMM$ are covectors we have $\Flat(X_1 \circ X_2) = \Flat(X_1) \cap \Flat(X_2)$.
\end{proof}

An element $i \in \III$ is a {\em coloop} of $\MMM$ if $X(i) = 0$ for all covectors $X \in \MMM$. The minimum element $\hat{0} \in \LLL(\MMM)$ is the set of all coloops. Recall that a covector $X \in \MMM$ is a {\em tope} if $\Supp(X) = \III$ and that $\TTT(\MMM)$ is the set of topes of $\MMM$. If $\MMM$ has a coloop then $\TTT(\MMM) = \varnothing$.

\subsection{Restriction and contraction}
We describe two constructions for building new conditional oriented matroids from $\MMM$. Let $F \in \LLL(\MMM)$ be a flat of $\MMM$. The {\em restriction} of $\MMM$ to $F$ is the family 
\begin{equation}
    \MMM \mid_F := \{ X \mid_F: F \to \{+,-,0\} \,:\, X \in \MMM \}
\end{equation}
of signed subsets of $F$. Given a flat $F \in \LLL(\MMM)$, the {\em contraction} of $\MMM$ to $F$ is the family 
\begin{equation}
    \MMM^F := \{ X \mid _{\III-F}: (\III -F) \to \{+,-,0\} \,:\, X(i) = 0 \text{ for all } i \in F \}
\end{equation}
of signed subsets of $\III -F$. It follows from \cite[Lem. 1]{BCK} that both $\MMM \mid_F$ and $\MMM^F$ are conditional oriented matroids.

A conditional oriented matroid $\MMM$ on $\III$ is an {\em oriented matroid} if the zero map $\zero: \III \to \{+,-,0\}$ is a covector in $\MMM$. If $\MMM$ is an oriented matroid on $\III$, we have an underlying (unoriented) matroid $\overline{\MMM}$ on $\III$ characterized by the poset (in this case lattice) of flats $\LLL(\MMM)$. If $\MMM$ is a conditional oriented matroid and $F$ is a flat of $\MMM$, the restriction $\MMM \mid_F$ is an oriented matroid on $F$.

For any flat $F \in \LLL(\MMM)$, contraction $\MMM^F$ does not have any coloops. The following basic combinatorial identity will have algebraic manifestations in Theorem~\ref{thm:big-locus-identification} and Corollary~\ref{cor:big-hilbert-series}.

\begin{lemma}
    \label{lem:tope-contraction-count}
    Let $\MMM$ be a conditional oriented matroid on the ground set $\III$. We have the equality
    \[
    \# \MMM = \sum_{F \in \LLL(\MMM)} \# \TTT(\MMM^F).
    \]
\end{lemma}

\begin{proof}
    If $X \in \MMM$ is a covector and $\Flat(X) = F$, then $X \mid_{\III-F}$ is a tope of $\MMM^F$.  The assignment $X \mapsto X \mid_{\III - F}$ defines a bijection
    \[
    \MMM \xrightarrow{\,\, \sim \, \,} \bigsqcup_{F \in \LLL(\MMM)} \TTT(\MMM^F)
    \]
    where $\sqcup$ stands for disjoint union.
\end{proof}

\subsection{Circuits and NBC sets}
The notion of a `circuit' of a conditional oriented matroid was introduced by Dorpalen-Barry, Proudfoot, and Wang \cite{DPW}. A signed set $X: \III \to \{+,-,0\}$ is a {\em circuit} of $\MMM$ if the following two conditions are satisfied.
\begin{enumerate}
    \item We have $X \circ Y \neq Y$ for all $Y \in \MMM$.
    \item If $Z: \III \to \{+,-,0\}$ is a proper signed subset of $X$, there exists $Y_0 \in \MMM$ such that $Z \circ Y_0 = Y_0$.
\end{enumerate}
A circuit $X$ of $\MMM$ is called {\em symmetric} if $-X$ is also a circuit. An element $i \in \III$ is a coloop if and only if $\MMM$ has a symmetric circuit with support $\{i\}$. Every circuit in an oriented matroid is symmetric.

Dorpalen-Barry, Proudfoot, and Wang established an important relationship between circuits and topes. Let $<$ be an arbitrary but fixed total order on the ground set $\III$ of a conditional oriented matroid $\MMM$. A subset $N \subseteq \III$ is called {\em no broken circuit (NBC)} if the following two conditions are satisfied.
\begin{enumerate}
    \item If $X$ is a circuit of $\MMM$, the containment $\Supp(X) \subseteq N$ does not hold.
    \item If $X$ is a symmetric circuit of $\MMM$, the containment $\Supp(X)^\circ \subseteq N$ does not hold, where $\Supp(X)^\circ$ is $\Supp(X)$ with its $<$-smallest element removed.
\end{enumerate}
Let $\NNN(\MMM)$ be the family of NBC sets. We have \cite[Prop. 4.2]{DPW} the important coincidence 
\begin{equation}
    \label{eq:nbc-tope}
    \# \TTT(\MMM) = \# \NNN(\MMM)
\end{equation}
of the number of topes of $\MMM$ with the number of NBC sets.

\subsection{Basic sets for flats}Let $\MMM$ be a conditional oriented matroid on $\III$; we have the following partially defined closure operation on subsets of $\III$. For any $C \subseteq \III$ which is contained in at least one flat of $\LLL(\MMM)$, define $\overline{C}$ to be the containment-minimal flat in $\LLL(\MMM)$ such that $C \subseteq \overline{C}$; such a flat $\overline{C}$ is guaranteed by Lemma~\ref{lem:flat-semilattice}.
If $C$ is not contained in any flat of $\MMM$, the symbol $\overline{C}$ is undefined.

Suppose $F \in \LLL(\MMM)$ is a flat of the conditional matroid $\MMM$ on $\III$.
A subset $B \subseteq F$ is {\em basic for $F$} if $\dots$
\begin{enumerate}
    \item  we have $\overline{B} = F$, and
    \item  if $C \subseteq B$ and $\overline{C} = F$ then $C = B$.
\end{enumerate}
A subset $C \subseteq \III$ is {\em nonbasic} if it is not basic for any flat $F$. Equivalently, a subset $C \subseteq \III$ is nonbasic if either $\dots$
\begin{itemize}
    \item $C \not\subseteq F$ for any flat $F \in \LLL(\MMM)$, or
    \item $C$ is contained in at least one flat and there is a proper subset $C' \subsetneq C$ such that $\overline{C'} = \overline{C}$.
\end{itemize}
If $\MMM$ is the conditional oriented matroid in Figure~\ref{fig:four-lines} and $F = \{1,2,3\}$, the basic sets for $F$ are $\{1,2\}, \{1,3\},$ and $\{2,3\}$. We have the following basic lemma about basic sets.

\begin{lemma}
    \label{lem:basic-lemma}
    Let $\MMM$ be a conditional oriented matroid on $\III$.
    \begin{enumerate}
        \item If $F \in \LLL(\MMM)$ is a flat, there is at least one basic set of $F$ and every basic set of $F$ has the same cardinality.
        \item If $F, F' \in \LLL(\MMM)$ are flats with $F' \subseteq F$, there exist basic sets $B$ for $F$ and $B'$ for $F'$ so that $B' \subseteq B$.
    \end{enumerate}
\end{lemma}

\begin{proof}
    The restriction $\MMM \mid_F$ of $\MMM$ to $F$ is an oriented matroid on the ground set $F$. Bases of the underlying oriented matroid $\overline{\MMM \mid_F}$ are basic sets of $F$; this proves (1). Since $F'$ is a flat of $\overline{\MMM \mid_F}$, (2) follows from Lemma~\ref{lem:flat-containment-bases}.
\end{proof}

\section{The small locus}
\label{sec:Small}

Throughout this section we fix a ground set $\III$ and conditional oriented matroid $\MMM$ on $\III$. Recall that a covector $X \in \MMM$ is a tope if $X(i) \neq 0$ for all $i$ and that $\TTT(\MMM)$ is the set of topes of $\MMM$.

\subsection{A locus for topes}
Consider the affine space $\FF^{\III \times \{+,-\}}$  with basis  \[\{ (i,s) \,:\, i \in \III, s \in \{+,-\} \}.\] For any tope $X \in \TTT(\MMM)$ we define a point $\yyy_X \in \FF^{\III \times \{+,-\}}$ with coordinates
\begin{equation}
    (\yyy_X)_{i,s} := \begin{cases}
        1 & \text{if $X(i) = s$,} \\
        0 & \text{otherwise.}
    \end{cases}
\end{equation}

\begin{defn}
    \label{def:small-locus}
     The {\em small locus} of $\MMM$ is the subset $\Ypoints_\MMM := \{ \yyy_X \,:\, X \in \TTT(\MMM) \}$ of $\FF^{\III \times \{+,-\}}$.
\end{defn}

    For example, let $\MMM$ be the conditional oriented matroid coming from Figure~\ref{fig:four-lines}. The small locus $\Ypoints_\MMM$ consists of $\# \TTT(\MMM) = 6$ points. The coordinates of these 6 points in \[\FF^{\III \times \{+,-\}} = \FF^{\{1,2,3,4\} \times \{+,-\}}\] are given by the rows of the following table.
    \[
    \begin{tabular}{c c c c c c c c}
        $(1,+)$ & $(1,-)$ & $(2,+)$ & $(2,-)$ & $(3,+)$ & $(3,-)$ & $(4,+)$ & $(4,-)$ \cr \hline
        1 & 0 & 1 & 0 & 1 & 0 & 1 & 0 \cr
        1 & 0 & 1 & 0 & 0 & 1 & 1 & 0 \cr
        0 & 1 & 1 & 0 & 0 & 1 & 1 & 0 \cr
        0 & 1 & 0 & 1 & 0 & 1 & 1 & 0 \cr
        0 & 1 & 0 & 1 & 1 & 0 & 1 & 0 \cr
        1 & 0 & 0 & 1 & 1 & 0 & 1 & 0 \cr
    \end{tabular}
    \]

\subsection{VG rings and orbit harmonics}
We write the coordinate ring $\FF^{\III \times \{+,-\}}$ as 
\[
S := \FF[y_i^+, y_i^- \,:\, i \in \III]
\]
so that the ideals $\II(\Ypoints_\MMM)$ and $\gr \, \II(\Ypoints_\MMM)$ live in $S$.  We explain how the orbit harmonics quotient ring $\RRR(\Ypoints_\MMM) = S/\gr \, \II(\Ypoints_\MMM)$ can be identified with the graded VG ring $\VVV_\MMM$ of $\MMM$ introduced by Dorpalen-Barry, Proudfoot, and Wang \cite{DPW}.

Recall that $\TTT(\MMM)$ denotes the set of topes of $\MMM$. Write $\VG_\MMM$ for the set of functions $\TTT(\MMM) \to \FF$. The set $\VG_\MMM$ attains the structure of an $\FF$-algebra under pointwise operations and is called the {\em Varchenko-Gelfand ring} of $\MMM$. 

For each $i \in \III$, we have the {\em Heaviside functions} $h_i^\pm \in \VG_\MMM$ given by
\begin{equation}
    h_i^+(X) = \begin{cases}
        1 & X(i) = + \\
        0 & X(i) = -
    \end{cases}  \quad \quad \quad
    h_i^-(X) = 1 - h_i^+(X) = \begin{cases}
        0 & X(i) = + \\
        1 & X(i) = -.
    \end{cases}
\end{equation}
These functions generate $\VG_\MMM$ as an $\FF$-algebra. We have a filtration $F_0 \subseteq F_1 \subseteq \cdots \subseteq \VG_\MMM$ where
\begin{equation}
    F_k = \{ \text{polynomials of degree $\leq k$ in the $h_i^\pm$}\} \subseteq \VG_\MMM.
\end{equation}
We write $\VVV_\MMM$ for the associated graded algebra, i.e.
\begin{equation}
    \VVV_\MMM = \bigoplus_{k \geq 0} F_k/F_{k-1}
\end{equation}
where $F_{-1}:=0$. The algebra $\VVV_\MMM$ is the {\em graded Varchenko-Gelfand ring} of $\MMM$.

\begin{proposition}
    \label{prop:small-locus-identification}
    The assignments $y_i^+ \mapsto h_i^+$ and $y_i^- \mapsto h_i^-$ induce isomorphisms of $\FF$-algebras
    \begin{equation}
        S/\II(\Ypoints_\MMM) \cong \VG_\MMM \quad \quad \text{and} \quad \quad \RRR(\Ypoints_\MMM) = S/\gr \, \II(\Ypoints_\MMM) \cong \VVV_\MMM.
    \end{equation}
\end{proposition}

\begin{proof}
    We have an identification $\FF[\Ypoints_\MMM] = \VG_\MMM$ under which the coordinate functions $y_i^\pm \in S$ evaluate in the same way as the Heaviside functions $h_i^\pm \in \VG_\MMM$.
\end{proof}

With Proposition~\ref{prop:small-locus-identification} in mind, explicit generators for the defining ideals $\II(\Ypoints_\MMM)$ and $\gr \, \II(\Zpoints_\MMM)$ of $\VG_\MMM$ and $\VVV_\MMM$ were given by Dorpalen-Barry, Proudfoot, and Wang.  Given a set of variables $\{z_1,\dots,z_n\}$, we write $e_d(z_1,\dots,z_n)$ for the elementary symmetric polynomial
\[
e_d(z_1,\dots,z_n) := \sum_{i_1 < \cdots<i_d} z_{i_1} \cdots z_{i_d}
\]
where $e_0(z_1,\dots,z_n) := 1$. In light of Proposition~\ref{prop:small-locus-identification}, the following result follows from \cite[Thm. 1.6]{DPW}.

\begin{theorem}
    \label{thm:small-ideal-generators} {\em (Dorpalen-Barry, Proudfoot, Wang \cite{DPW})} The ideal $\II(\Ypoints_\MMM) \subseteq S$ is generated by $\dots$
    \begin{enumerate}
        \item the product $y_i^+ y_i^-$ and the difference $y_i^+ + y_i^- - 1$ for each $i \in \III$, and
        \item the product $\prod_{i \in \Supp(X)} y_i^{X(i)}$ for each circuit $X$ of $\MMM$.
    \end{enumerate}
    The ideal $\gr \, \II(\Ypoints_\MMM) \subseteq S$ is generated by $\dots$
    \begin{enumerate}
        \item every degree two monomial in $\{y_i^+, y_i^-\}$ for $i \in \III$,
        \item the sum $y_i^+ + y_i^-$ for each $i \in \III$,
        \item the product $\prod_{i \in \Supp(X)} y_i^{X(i)}$ for each circuit $X$ of $\MMM$, and
        \item the elementary symmetric polynomial $e_{s-1}(y_i^{X(i)} \,:\,i \in \Supp(X))$ for each symmetric circuit $X$ of $\MMM$ with support size $\# \Supp(X) = s$.
    \end{enumerate}
\end{theorem}

It follows from results in \cite[Sec. 5]{DPW} that the set $\{ \prod_{i \in N} y_i^+ \,:\, N \in \NNN(\MMM) \}$ of monomials in the $y_i^+$ indexed by NBC sets descends to a vector space basis of either $S/\II(\Ypoints_\MMM)$ or $S/\gr \, \II(\Ypoints_\MMM) = \RRR(\Ypoints_\MMM) \cong \VVV_\MMM$. Our `big' VG quotients will admit a basis of a similar flavor.

\section{The big locus}
\label{sec:Big}

Fix a conditional matroid $\MMM$ on the ground set $\III$. We work in the affine space $\FF^{\III \times \{+,-,0\}}$ with basis
\[\{(i,s) \,:\, i \in \III, \, s \in \{+,-,0\}\}.\] 
This section studies a locus $\Zpoints_\MMM$ whose points are indexed by covectors, not just topes. While the results in the last section were orbit harmonics reformulations of those in \cite{DPW}, the results in this section are new.

\subsection{A locus for covectors} For any covector $X \in \MMM$ we define a point $\zzz_X \in \FF^{\III \times \{+,-,0\}}$ with coordinates
\begin{equation}
    (\zzz_X)_{i,s} := \begin{cases}
        1 & X(i) = s, \\
        0 & \text{otherwise.}
    \end{cases}
\end{equation}

\begin{defn}
    \label{def:big-locus}
     The {\em big locus} of $\MMM$ is the subset $\Zpoints_\MMM := \{ \zzz_X \,:\, X \in \MMM \}$ of $\FF^{\III \times \{+,-,0\}}$.
\end{defn}

If $\MMM$ is the conditional oriented matroid coming from Figure~\ref{fig:four-lines}, the locus $\Zpoints_\MMM$ contains a point $\zzz_X \in \FF^{\III \times \{+,-,0\}}$ for each covector $X \in \MMM$, so that $\# \Zpoints_\MMM = \# \MMM = 13$. The coordinates of $\zzz_X$ where $X \in \MMM$ is the covector coming from the bolded face in Figure~\ref{fig:four-lines} are as follows.
\[
\begin{tabular}{cccccccccccc}
    $(1,+)$ & $(1,-)$ & $(1,0)$ & $(2,+)$ & $(2,-)$ & $(2,0)$ & $(3,+)$ & $(3,-)$ & $(3,0)$ & $(4,+)$ & $(4,-)$ & $(4,0)$ \cr \hline
    0 & 0 & 1 & 1 & 0 & 0 & 0 & 1 & 0 & 1 & 0 & 0 
\end{tabular}
\]

 We write the coordinate ring of $\FF^{\III \times \{+,-,0\}}$ as 
\[
    \widehat{S} := \FF[y_i^+,y_i^-,z_i \,:\, i \in \III]
\]
so that the ideals $\II(\Zpoints_\MMM)$ and $\gr \, \II(\Zpoints_\MMM)$ live in $\widehat{S}$. The main goal of this section is to understand the structure of $\RRR(\Zpoints_\MMM) = \widehat{S}/\gr \, \II(\Zpoints_\MMM).$

\subsection{The big VG ring} We define the {\em big Varchenko-Gelfand ring} of $\MMM$ to be the algebra $\widehat{\VG}_\MMM$ of functions $\MMM \to \FF$ with pointwise operations. This algebra was considered in the affine case by Varchenko and Gelfand \cite{VG}; see also \cite{GR}. 

As with the `small' VG ring $\VG_\MMM$, we have the  Heaviside functions $\{h_i^+, h_i^- \,:\, i \in \III\}$ defined by 
\[
h_i^+(X) = \begin{cases}
    1 & X(i) = + \\
    0 & X(i) = - \text{ or } 0
\end{cases} \quad \quad \quad 
h_i^-(X) = \begin{cases}
    1 & X(i) = - \\
    0 & X(i) = + \text{ or } 0.
\end{cases}
\]
Observing that the indicator function $h^0_i(X)$ for whether $X(i)=0$ satisfies \[h_i^0 = 1 - h_i^+ - h_i^-,\] we see that the set $\{h_i^\pm \,:\, i \in \III \}$ generates $\widehat{\VG}_\MMM$ as an algebra.
This gives rise to a filtration $F_0 \subseteq F_1 \subseteq \cdots \subseteq \widehat{VG}_\MMM$ where
\begin{equation}
    F_k := \{ \text{polynomials in $h_i^\pm$ of degree $\leq k$} \} \subseteq \widehat{\VG}_\MMM.
\end{equation}
The {\em graded big Varchenko-Gelfand ring} $\widehat{\VVV}_\MMM$ is  the associated graded ring
\begin{equation}
    \widehat{\VVV}_\MMM := \bigoplus_{k \geq 0} F_k/F_{k-1}
\end{equation}
where $F_{-1} := 0$. Both $\widehat{\VG}_\MMM$ and $\widehat{\VVV}_\MMM$ have orbit harmonics interpretations.

\begin{proposition}
    \label{prop:big-orbit-harmonics-interpretations}
    The assignments $y_i^+ \mapsto h_i^+, y_i^- \mapsto h_i^-,$ and $z_i \mapsto 1 - h_i^+ - h_i^-$ give an identification
    \begin{equation}
    \FF[\Zpoints_\MMM] = \widehat{S}/\II(\Zpoints_\MMM) \cong \widehat{\VG}_\MMM.
    \end{equation}
    This induces the identification of associated graded rings 
    \begin{equation}
        \RRR(\Zpoints_\MMM) = \widehat{S}/\gr \, \II(\Zpoints_\MMM) \cong \widehat{\VVV}_\MMM.
    \end{equation}
\end{proposition}

\begin{proof}
    Points $\zzz_X$ in the big locus $\Zpoints_\MMM$ are indexed by covectors $X \in \MMM$. The first isomorphism follows by considering how the coordinate functions $y_i^+,y_i^-,z_i$ evaluate on these covectors. In order to prove the second isomorphism, we observe that the filtration defining $\widehat{\VVV}_\MMM$ satisfies
    \[
    F_k = \{ \text{polynomials in $h_i^\pm, (1 - h_i^+ - h_i^-)$ of degree $\leq k$} \} \subseteq \widehat{\VG}_\MMM
    \]
    which agrees with the filtration on $\widehat{S}$ with $\deg(y_i^\pm) = \deg(z_i) = 1$ defining $\RRR(\Zpoints_\MMM)$.
\end{proof}

In this section we describe generating sets for the ideals $\II(\Zpoints_\MMM)$ and $\gr \, \II(\Zpoints_\MMM)$. The quotient ring $\RRR(\Zpoints_\MMM)$  admits an interesting filtration indexed by the poset of flats $\LLL(\MMM)$. We digress to state some general results on filtrations indexed by posets.

\subsection{$P$-filtrations} Let $P$ be a finite poset with a unique minimum $\hat{0}$ and $V$ be a finite-dimensional $\FF$-vector space. A {\em $P$-filtration} of $V$ consists of a subspace $V_x \subseteq V$ for each $x \in P$ such that $V_y \subseteq V_x$ whenever $x \leq y$. (Note the reversal of containment of subspaces and order in $P$.) We require that $V_{\hat{0}} = V$. 
 If $V$ is a graded vector space, the $P$-filtration $\{V_x \,:\, x \in P \}$ is {\em graded} if each $V_x$ is a graded subspace of $V$.

Given any $P$-filtration of $V$ and any $x \in P$, we have the subspaces
\[
V_{\geq x} := \sum_{y \geq x} V_y = V_x \quad \text{and} \quad V_{> x} := \sum_{y > x} V_y
\]
together with the quotient space
\[
V_{= x} := V_{\geq x}/V_{> x}.
\]
In the following result, for polynomials $f(q),g(q) \in \ZZ_{\geq 0}[q]$  we write $f(q) \leq g(q)$ to mean that $g(q) - f(q)$ has nonnegative coefficients. 

\begin{lemma}
    \label{lem:P-filtration}
    Suppose $\{V_x \,:\, x \in P\}$ is a $P$-filtration of $V$. We have the dimension inequality
    \begin{equation}
        \dim V \leq \sum_{x \in P} \dim(V_{=x}).
    \end{equation}
    Furthermore, if $V$ is a graded vector space and $\{V_x \,:\, x \in P\}$ is a graded $P$-filtration, we have 
    \begin{equation}
        \Hilb(V;q) \leq \sum_{x \in P} \Hilb(V_{=x};q).
    \end{equation}
\end{lemma}

\begin{proof}
    Assume that $P$ has $n$ elements and let $f: P \to [n]$ be an order-preserving bijection where $[n] := \{1,\dots,n\}$ has its usual order. We necessarily have $f(\hat{0})= 1$. Define a descending chain of subspaces $W_1 \supseteq \cdots \supseteq W_n$ of $V$ by
    \begin{equation}
    W_i := V_{f^{-1}(i)} + V_{f^{-1}(i+1)} + \cdots + V_{f^{-1}(n)}.
    \end{equation}
    Since $f: P \to [n]$ is order-preserving, for all $x \in P$ we have a canonical surjection
    \begin{equation}
    \label{eq:canonical-surjection}
    V_{=x} = V_{\geq x}/V_{> x} \twoheadrightarrow W_{f(x)}/W_{f(x)+1} \quad \text{for $1 \leq i \leq n$}
    \end{equation}
    where $W_{n+1} := 0$ so that
    \begin{equation}
    \sum_{i=1}^n \dim(W_i/W_{i+1}) \leq \sum_{x \in P} \dim V_{=x}.
    \end{equation}
    The assumption $V_{\hat{0}} = P$ implies $W_1 = V$ so that $\sum_{i=1}^n \dim(W_i/W_{i+1}) = \dim V$ and the ungraded statement is proven. For the graded statement, observe that the canonical maps in \eqref{eq:canonical-surjection} are homogeneous.
\end{proof}

Inequality can occur in Lemma~\ref{lem:P-filtration}. For example, let $V = \FF^2$ and consider the $P$-filtration shown below.
\begin{center}
    \begin{tikzpicture}[scale = 1]
        \node at (0,0) {$\FF^2$};
        \node at (-2,1) {$\FF \cdot (1,0)$};
        \node at (2,1) {$\FF \cdot (0,1)$};
        \node at (0,1) {$\FF \cdot (1,1)$};

        \draw [-] (0,0.2) -- (0,0.7);
        \draw [-] (0.3,0.2) -- (1.3,0.8);
        \draw [-] (-0.3,0.2) -- (-1.3,0.8);
    \end{tikzpicture}
\end{center}
In the cases considered in this paper, we will have $P = \LLL(\MMM)$ for some conditional oriented matroid $\MMM$ and equality will occur when applying Lemma~\ref{lem:P-filtration}.

When equality holds in Lemma~\ref{lem:P-filtration}, we have an equivariant enhancement. Suppose the vector space $V$ carries the linear action of a finite group $G$ and let $\{ V_x \,:\, x \in P \}$ be a $P$-filtration of $V$. Suppose the group $G$ also acts on the poset $P$ by order-preserving automorphisms, and these actions are compatible in the sense that 
\begin{equation}
g \cdot V_x = V_{g \cdot x} \quad \text{ for all $g \in G$ and $x \in P$.}
\end{equation}
In this setting we say that the $P$-filtration $\{V_x \,:\, x \in P \}$ is {\em $G$-equivariant}.
For any $x \in P$, let $G_x \subseteq G$ be the stabilizer subgroup
\begin{equation}
    G_x := \{ g \in G \,:\, g \cdot x = x \} .
\end{equation}
Since the action of $G$ on $P$ is order-preserving, the vector spaces $V_x, V_{\geq x}, V_{> x},$ and $V_{= x}$ all carry actions of the subgroup $G_x$. In the following lemma we write $P/G$ for the quotient poset of $G$-orbits in $P$ and $\Ind_{G_x}^G(-)$ for the induction functor from $G_x$-modules to $G$-modules.

\begin{lemma}
    \label{lem:P-filtration-equivariant}
    Let $\{V_x \,:\, x \in P \}$ be a $G$-equivariant $P$-filtration of $V$ where $\# G \in \FF^\times$. Assume that 
    \begin{equation}
        \dim V = \sum_{x \in P} \dim V_{=x}.
    \end{equation}
    There exists an isomorphism of $G$-modules 
    \begin{equation}
    \label{eq:general-character-equality}
        V \cong_G \bigoplus_{[x] \in P/G} \Ind_{G_x}^G (V_{=x}).
    \end{equation}
    If $V$ is a graded vector space and the $P$-filtration is graded, this may be taken to be an isomorphism of graded $G$-modules.
\end{lemma}

\begin{proof}
    For each $x \in P$, let $\BBB_x \subseteq V_x$ be a set of polynomials which descends to a vector space basis of $V_{=x} = V_x/V_{> x}$. Choose these polynomials in such a way that $g \cdot \BBB_x = \BBB_{g \cdot x}$ for all $g \in G$ and $x \in P$.\footnote {This may be done by choosing such a set $\BBB_x \subseteq V_x$ for $x$ belonging to a transversal of the orbit set $G/P$, and then defining $\BBB_{g \cdot x} := g \cdot \BBB_x$ for all $g \in G$.} The assumption $\dim V = \sum_{x \in P} \dim V_{=x}$ implies that the union $\BBB := \bigsqcup_{x \in X} \BBB_x$ is disjoint and that $\BBB$ is a basis of $V$.

    Consider the representing matrices for the action of $G$ on $V$ with respect to the basis $\BBB$. We have a stratification 
    \[
    \BBB = \bigsqcup_{[x] \in P/G} \BBB_{[x]} \quad \text{where} \quad \BBB_{[x]} := \bigsqcup_{x' \in [x]} \BBB_{x'}
    \]
    of $\BBB$ indexed by the quotient poset $P/G$. The action of $G$ on $V$ triangular with respect to this stratification and the partial order on $P/G$. The diagonal block corresponding to $\BBB_{[x]}$ affords the induced representation $\Ind_{G_x}^G(V_{=x})$. It follows that the representations on either side of \eqref{eq:general-character-equality} have the same character. Since $\# G$ is a unit in $\FF$, these representations are isomorphic. In the graded setting, we repeat this argument with the assumption that the sets $\BBB_x$ are homogeneous.
\end{proof}

When $\# G$ is not a unit in $\FF$, Lemma~\ref{lem:P-filtration-equivariant} holds with the weaker conclusion that the $G$-modules $V$ and $\bigoplus_{[x] \in P/G} \Ind_{G_x}^G(V_{=x})$ are {\em Brauer isomorphic}. That is, these modules have the same composition factors.

\subsection{Middle quotients for the big locus} 
The ring $\RRR(\Zpoints_\MMM)$ admits a useful $\LLL(\MMM)$-filtration where $\LLL(\MMM)$ is the poset of flats of $\MMM$.
To describe this filtration, we introduce an intermediate family of quotient rings
\[
\widehat{S} \twoheadrightarrow \widehat{S}/\widetilde{I}_\MMM \twoheadrightarrow \RRR(\Zpoints_\MMM).
\]
The generators for the ideal $\widetilde{I}_\MMM$ defining our middle quotient rings are as follows. As the main purpose of these middle rings is to introduce $\LLL(\MMM)$-filtrations, these generators only involve the variables $\{z_i \,:\, i \in \III \}$.

\begin{defn}
\label{def:tilde-ideal}
Let $\MMM$ be a conditional oriented matroid on $\III$. Define $\widetilde{I}_\MMM \subseteq \widehat{S}$ to be the ideal with the following generators.
\begin{enumerate}
    \item The product $\prod_{c \in C} z_c$ for any nonbasic set $C \subseteq \III$.
    \item The difference $\prod_{b \in B} z_b - \prod_{b' \in B'} z_{b'}$ for any two basic sets $B,B'$ for the same flat $F$.
    \item The element $z_i^2$ for each $i \in \III$.
\end{enumerate}
\end{defn}

\begin{remark}
    \label{rmk:HW}
    The ideals $\widetilde{I}_\MMM$ arise in geometry. Let $L \subseteq \CC^n$ be a linear subspace. Ardila and Boocher introduced \cite{AB} the {\em matroid Schubert variety} as the Zariski closure $Y_L \subseteq (\PP^1)^n$ of the image of 
    \[
    L \hookrightarrow \CC^n \hookrightarrow (\PP^1)^n.
    \]
    Intersecting the coordinate hyperplanes of $\CC^n$ with $L$ gives rise to a matroid $\MMM$ on the ground set $\III = [n]$. Let $z_i \in H^2(Y_L)$ be the first Chern class of the tautological line bundle over the $i^{th}$ factor in $(\PP^1)^n$, pulled back along $Y_L \hookrightarrow (\PP^1)^n.$
    Huh and Wang proved \cite{HW} that the cohomology ring $H^*(Y_L)$ is generated by $z_1,\dots,z_n$ with relations as in Definition~\ref{def:tilde-ideal}. We may therefore regard the quotient $\widehat{S}/\widetilde{I}_\MMM$ as the cohomology of $Y_L$ with coefficients in $\FF[y_i^\pm \,:\, i \in \III]$.
\end{remark}

In the example of Figure~\ref{fig:four-lines}, we have $\III = \{1,2,3,4\}$ so that
\[ \widehat{S} = \FF[y_1^+,y_1^-,z_1,\dots, y_4^+,y_4^-,z_4].\]
The minimal nonbasic sets are $\{1,2,3\}$ and $\{4\}$.
The ideal $\widetilde{I}_\MMM \subseteq \widehat{S}$ is given by
\[ \widetilde{I}_\MMM = (z_1 \cdot z_2 \cdot z_3, \,  z_4,  \,  z_1 \cdot z_2 - z_1 \cdot z_3, \,   z_1 \cdot z_2 - z_2 \cdot z_3, \,   z_1 \cdot z_3 - z_2 \cdot z_3, \, z_1^2, \, z_2^2, \, z_3^2, \, z_4^2).\]

Lemma~\ref{lem:basic-lemma} (1) implies that  $\widetilde{I}_\MMM$ is a homogeneous ideal in $\widehat{S}$. The following result gives the desired sequence $\widehat{S} \twoheadrightarrow \widehat{S}/\widetilde{I}_\MMM \twoheadrightarrow \RRR(\Zpoints_\MMM)$ of canonical surjections.

\begin{lemma}
    \label{lem:tilde-ideal-containment}
    We have the containment $\widetilde{I}_\MMM \subseteq \gr \, \II(\Zpoints_\MMM)$ of ideals in $\widehat{S}$.
\end{lemma}

\begin{proof}
    We check that the three kinds of generators in Definition~\ref{def:tilde-ideal} lie in $\gr \, \II(\Zpoints_\MMM)$.

    (1) Let $C \subseteq \III$ be a nonbasic set. If $C\not\subseteq F$ for all flats $F \in \LLL(\MMM)$, one has the membership
    \[
    \prod_{c \in C} z_c \in \II(\Zpoints_\MMM) \quad \text{so that } \quad \prod_{c \in C} z_c \in \gr \, \II(\Zpoints_\MMM).
    \]
    Otherwise, let $F$ be the containment-minimal flat in $\LLL(\MMM)$ so that $C \subseteq F$. Let $B \subsetneq C$ be a (necessarily proper) subset of $C$ such that $C$ is basic for $F$. We claim that  
    \[
    \prod_{c \in C} z_c - \prod_{b \in B} z_b \in \II(\Zpoints_\MMM).
    \]
    Indeed, if $X \in \MMM$ is a covector, the above difference evaluates on $\zzz_X$ to $1-1=0$ when $F \subseteq \Flat(X)$ and $0-0 = 0$ when $F \not\subseteq \Flat(X)$.
    Since $\# C > \# B$, taking the highest degree component gives the desired membership $\prod_{c \in C} z_c \in \gr \, \II(\Zpoints_\MMM)$.

    (2) Suppose $B$ and $B'$ are basic for the same flat $F \in \LLL(\MMM)$. By the same reasoning as in (1) one has 
    \[
    \prod_{b \in B} z_b - \prod_{b \in B'} z_{b'} \in \II(\Zpoints_\MMM).
    \]
    Lemma~\ref{lem:basic-lemma} (1) guarantees that the polynomial $\prod_{b \in \BBB} z_b - \prod_{b \in \BBB'} z_{b'} \in \hat{S}$ is homogeneous, and so lies in $\gr \, \II(\Zpoints_\MMM)$.

    (3) For any $i \in \III$ we have $z_i(z_i-1) \in \II(\Zpoints_\MMM)$ so that $z_i^2 \in \gr \, \II(\Zpoints_\MMM)$.
\end{proof}

 We introduce the following family of elements of $\widehat{S}/\widetilde{I}_\MMM$  indexed by flats in $\LLL(\MMM)$. These elements will be used to define our $\LLL(\MMM)$-filtrations.

\begin{defn}
    \label{def:flat-elements}
    Let $\MMM$ be a conditional oriented matroid on $\III$ and let $F \in \LLL(\MMM)$ be a flat of $\MMM$. We define an element $z_F \in \tilde{R}_\MMM$ by 
    \[
    z_F := \prod_{b \in B} z_b + \widetilde{I}_\MMM \in \widehat{S}/\widetilde{I}_\MMM.
    \]
    where $B$ is any basic set for $F$.
\end{defn}

If $\MMM$ is as in Figure~\ref{fig:four-lines} and $F = \{1,2,3\} \in \LLL(\MMM)$, one has the element
\[
z_F = z_1 \cdot z_2 + \widetilde{I}_\MMM = z_1 \cdot z_3 + \widetilde{I}_\MMM = z_2 \cdot z_3 + \widetilde{I}_\MMM \in \widehat{S}/\widetilde{I}_\MMM.
\]
Lemma~\ref{lem:basic-lemma} (1) and Definition~\ref{def:tilde-ideal} (2) imply that $z_F$ is well-defined in general (i.e. independent of the choice of $B$). The following divisibility property of the $z_F$ will give rise to $\LLL(\MMM)$-filtrations.

\begin{lemma}
\label{lem:z-divisibility}
Suppose $F_1 \subseteq F_2$ are flats of $\MMM$. We have the divisibility $z_{F_1} \mid z_{F_2}$ in $\widehat{S}/\widetilde{I}_\MMM$.
\end{lemma}

\begin{proof}
    This follows from Lemma~\ref{lem:basic-lemma} (2) and Definition~\ref{def:flat-elements}.
\end{proof}

\subsection{The big orbit harmonics quotient} In this subsection we give a generating set of the ideal $\gr \, \II(\Zpoints_\MMM) \subseteq \hat{S}$ coming from the big locus $\Zpoints_\MMM$.  To do this, we need a lemma about the values functions in $\widehat{S}$ take on points in $\Zpoints_\MMM$ in the presence of symmetric circuits.

\begin{lemma}
    \label{lem:two-values-attained-big}
    Let $X: \III \to \{+,-,0\}$ be a symmetric circuit of $\MMM$ whose support has cardinality $s = \# \Supp(X) > 1$. Let $J$ be a nonempty and proper subset of the support of $X$, i.e. 
    \[ \varnothing \subsetneq J \subsetneq \Supp(X).\]
    For $i \in \Supp(X)$, define an element $\tilde{y}_i \in \widehat{S}$ by the rule
    \[
    \tilde{y}_i := \begin{cases}
        y_i^{X(i)} + z_i &  i \in J, \\
        y_i^{X(i)} &  i \in \Supp(X) - J.
    \end{cases}
    \]
    Let $Y \in \MMM$ be a covector. There exist elements $i_0, i_1 \in \Supp(X)$ such that $\tilde{y}_{i_0}(\zzz_Y) = 0$ and $\tilde{y}_{i_1}(\zzz_Y) = 1$.
\end{lemma}

\begin{proof}
    Replacing $\MMM$ by its restriction $\MMM \mid_{\Supp(X)}$ to $\Supp(X)$ if necessary, we reduce to the case $\Supp(X) = \III = [n]$. Switching signs and rearranging coordinates if necessary, we assume further that
    \[
    X = (\overbrace{+,\dots,+}^n) \quad \text{and $J = [m]$ for some $0 < m < n$.}
    \]
    Let $Y \in \MMM$ be a covector. We argue by contradiction as follows.

    Assume first that $\tilde{y}_i = 1$ for all $1 \leq i \leq n$. Then $Y$ has the form
    \[
    Y = (c_1,\dots,c_m, \overbrace{+,\dots,+}^{n-m})
    \]
    where $c_1,\dots, c_m \in \{+,0\}$. Rearranging coordinates if needed, we may assume  
    \[
    Y = (\overbrace{0,\dots,0}^a, \overbrace{+,\dots,+}^b, \overbrace{+,\dots,+}^{n-m})
    \]
    where $a + b = m$.
    Let $X': \III \to \{+,-,0\}$ be the signed set
    \[
    X' := (\overbrace{+,\dots,+}^a, \overbrace{0,\dots,0}^{n-a}).
    \]
    Since $m < n$ and $X$ is a circuit, support minimality implies that there exists $Y' \in \MMM$ so that $X' \circ Y' = Y'$. The covector $Y'$ has the form
    \[
    Y' = (\overbrace{+,\dots,+}^a, d_{a+1},\dots,d_n)
    \]
    where $d_i \in \{+,-,0\}$. Since $\MMM$ is closed under composition, we have $Y \circ Y' = (+,\dots,+) \in \MMM$. But then $X \circ (Y \circ Y') = Y \circ Y'$, which contradicts the assumption that $X$ is a circuit.

    Now assume that $\tilde{y}_i = 0$ for all $1 \leq i \leq n$. Rearranging coordinates if necessary as before, we amy assume
    \[
    Y = (\overbrace{-,\dots,-}^m, \overbrace{-,\dots,-}^a, \overbrace{0,\dots,0}^b).
    \]
    where $a+b = n-m$. Since $X$ is symmetric, we know that $-X$ is a circuit. Since $m > 0$, one has $b < n$ and there exists a covector $Y' \in \MMM$ so that 
    \[
    (\overbrace{0,\dots,0}^{n-b}, \overbrace{-,\dots,-}^b) \circ Y' = Y'.
    \]
    The covector $Y' \in \MMM$ has the form
    \[
    Y' = (d_1,\dots,d_{n-b}, \overbrace{-,\dots,-}^b)
    \]
    where $d_i \in \{+,-,0\}$. But then $Y \circ Y' = (-,\dots,-) \in \MMM$ and $(-X) \circ (Y \circ Y') = Y \circ Y'$ so that $-X$ is not a circuit of $\MMM$.
\end{proof}

We are ready to describe a generating set of $\gr \, \II(\Zpoints_\MMM)$. Given a flat $F \in \LLL(\MMM)$, recall that the contraction $\MMM^F$ is the conditional oriented matroid on the ground set $\III - F$ with covectors
\[
\MMM^F = \{ X \mid_{\III-F} \,:\, X \in \MMM \text{ and } X(i) = 0 \text{ for all } i \in F \}.
\]
Also recall the element $z_F \in \widehat{S}/\widetilde{I}_\MMM$ of Definition~\ref{def:flat-elements}.

\begin{defn}
    \label{def:I-ideal}
    Let $\pi: \widehat{S} \twoheadrightarrow \widehat{S}/\widetilde{I}_\MMM$ be the canonical surjection. Define
    $\widehat{I}_\MMM \subseteq \widehat{S}$ to be the full preimage under $\pi$ of the ideal with the following generators.
    \begin{enumerate}
        \item All degree two monomials in the set $\{y_i^+,y_i^-,z_i\}$ for each $i \in \III$.
        \item The sum $y_i^+ + y_i^- + z_i$ for each $i \in \III$.
        \item For each flat $F \in \LLL(\MMM)$ and each $i \in F$, the products
        \[
        z_F \cdot y_i^+ \quad \text{and} \quad z_F \cdot y_i^-.
        \]
        \item For each flat $F \in \LLL(\MMM)$ and each circuit $X: (\III-F) \to \{+,-,0\}$ of $\MMM^F$, the product
        \[
        z_F \cdot \prod_{i \in \Supp(X)} y_i^{X(i)}.
        \]
        \item For each flat $F \in \LLL(\MMM)$ and each symmetric circuit $X: (\III-F) \to \{+,-,0\}$ of $\MMM^F$, the element
        \[
        z_F \cdot e_{s-1}(\tilde{y}_i \,:\, i \in \Supp(X))
        \]
        where $\varnothing \subsetneq J \subsetneq \Supp(X)$ is a fixed nonempty and proper subset of $\Supp(X)$ and $\tilde{y}_i$ is defined using $J$ as in Lemma~\ref{lem:two-values-attained-big}.
    \end{enumerate}
\end{defn}

We make two remarks on Definition~\ref{def:I-ideal} (5). First, since the contraction $\MMM^F$ does not have any coloops, any symmetric circuit $X$ of $\MMM^F$ satisfies $\# \Supp(X) > 1$, so the set $\Supp(X)$ has a nonempty and proper subset $J$. Second, although different choices of $J$ could give rise to different polynomials $z_F \cdot e_{s-1}(\tilde{y}_i \,:\, i \in \Supp(X))$, in Theorem~\ref{thm:big-locus-identification} we will show that the ideal $\widehat{I}_\MMM$ is independent of the choice of $J$ (and in fact equal to $\gr \, \II(\Zpoints_\MMM)$).

\begin{example}
 Let $\MMM$ be the conditional oriented matroid on $\III = \{1,2,3,4\}$ in Figure~\ref{fig:four-lines}. Every flat $F \in \LLL(\MMM)$ gives rise to generators of $\widehat{I}_\MMM$ as in Definition~\ref{def:I-ideal} (3), (4), (5). We calculate these generators when $F = \{2\}$ and $F = \{1,2,3\}$.
 
 Consider the flat $F = \{2\} \in \LLL(\MMM)$. The contraction $\MMM^F$ has ground set $\III - F = \{1,3,4\}$ and consists of the following three covectors.
\[
\begin{tabular}{c | ccc}
 & 1 & 3 & 4 \\ \hline
 $X_1$ & $+$ & $+$ & $+$ \\
 $X_2$ & $0$ & $0$ & $+$ \\
 $X_3$ & $-$ & $-$ & $+$ 
\end{tabular}
\]
The only basic set for $F$ is $\{2\}$, so we have $z_F = z_2$. The generators of $\widehat{I}_\MMM$ coming from (3) are 
$z_2 \cdot y_2^+$ and $z_2 \cdot y_2^-$; these generators are present in (1). The generator of $\widehat{I}_\MMM$ coming from (4) arises from the nonsymmetric circuit $(0,0,-)$ of $\MMM^F$ and is given by $z_2 \cdot y_4^-$; this is implied by the generator $y_4^-$ coming from the flat $\varnothing$. Finally, a generator of $\widehat{I}_\MMM$ coming from (5) arises from the symmetric circuit $X = (+,-,0)$  with $\Supp(X) = \{1,3\}$ and $\JJJ = \{3\}$; this generator is given by
\[
z_2 \cdot (\tilde{y}_1 + \tilde{y}_3) = z_2 \cdot (y_1^+ + (y_3^- + z_3)).
\]

Now consider the flat $F = \{1,2,3\} \in \LLL(\MMM)$. The contraction $\MMM^F$ has ground set $\III - F = \{4\}$ and consists of the single covector $(+)$. The set $\{1,2\}$ is basic for $F$ so that $z_F = z_1 \cdot z_2$. Definition~\ref{def:I-ideal} (3) contributes the generators
\[
z_1 \cdot z_2 \cdot y_1^+, \quad z_1 \cdot z_2 \cdot y_1^-, \quad z_1 \cdot z_2 \cdot y_2^+, \quad z_1 \cdot z_2 \cdot y_2^-, \quad z_1 \cdot z_2 \cdot y_3^+, \quad z_1 \cdot z_2 \cdot y_3^-
\]
of $\widehat{I}_\MMM$. The last two of these generators are not implied by (1). Definition~\ref{def:I-ideal} (4) and the nonsymmetric circuit $(-)$ give the generator $z_1 \cdot z_2 \cdot y_4^-$; this is implied by the generator $y_4^-$ coming from the flat $\varnothing$. Since $\MMM^F$ does not have any symmetric circuits, there are no generators coming from Definition~\ref{def:I-ideal} (5).
\end{example}

The above example shows that there can be redundancy in the generating set of $\widehat{I}_\MMM$. It may be interesting to describe a generating set of $\widehat{I}_\MMM$ with less redundancy. We prove that $\widehat{S}/\widehat{I}_\MMM$ projects onto $\RRR(\Zpoints_\MMM)$.

\begin{lemma}
    \label{lem:big-ideal-containment}
    One has the containment  $\widehat{I}_\MMM \subseteq \gr \, \II(\Zpoints_\MMM)$ of ideals in $\widehat{S}$.
\end{lemma}

\begin{proof}
    We show that each generator of $\widehat{I}_\MMM$ lies in $\gr \, \II(\Zpoints_\MMM)$. The five types of generators in Definition~\ref{def:I-ideal} are handled as follows.

    (1) For any $i \in \III$, the definition of $\Zpoints_\MMM$ yields the memberships
    \[
    y_i^+ \cdot (y_i^+ - 1), \quad y_i^- \cdot (y_i^- - 1), \quad z_i \cdot (z_i - 1), \quad y_i^+ \cdot y_i^-, \quad y_i^+ \cdot z_i, \quad y_i^- \cdot z_i \quad \in \quad \II(\Zpoints_\MMM).
    \]
    Taking highest degree components, we see that any degree 2 monomial in $\{y_i^+,y_i^-,z_i\}$ lies in $\gr \, \II(\Zpoints_\MMM)$.

    (2) For any $i \in \III$, the definition of $\Zpoints_\MMM$ implies that $y_i^+ +  y_i^- + z_i - 1 \in \II(\Zpoints_\MMM)$. Taking the top degree component gives $y_i^+ + y_i^- + z_i \in \gr \, \II(\Zpoints_\MMM)$.

    (3) Let $F \in \LLL(\MMM)$ be a flat and fix a basic set $B \subseteq F$. If $i \in F$, the definition of $\Zpoints_\MMM$ implies 
    \begin{equation}
        \prod_{b \in B} z_b \cdot y_i^+, \quad  \prod_{b \in B} z_b \cdot y_i^- \quad \in \quad \II(\Zpoints_\MMM)
    \end{equation}
    so that these homogeneous polynomials also lie in $\gr \, \II(\Zpoints_\MMM)$.

    (4) Let $F \in \LLL(\MMM)$ be a flat and fix a basic set $B \subseteq F$. Let $X: (\III -F) \to \{+,-,0\}$ be a circuit of $\MMM^F$. We establish the membership  
    \begin{equation}
    \label{eq:desired-membership-three}
    \prod_{b \in B} z_b \cdot \prod_{i \in \Supp(X)} y_i^{X(i)} \in \gr \, \II(\Zpoints_\MMM).
    \end{equation}
    Viewed as a function on the locus $\Zpoints_\MMM$, one has the evaluation
    \begin{equation}
    \label{eqn:face-element-evaluation}
    \prod_{b \in B} z_b : \zzz_X \mapsto \begin{cases}
        1 & \text{$X(i) = 0$ for all $i \in F$}, \\
        0 & \text{otherwise,}
    \end{cases}
    \end{equation}
    so that \eqref{eq:desired-membership-three} vanishes on any locus point $\zzz_X$ corresponding to a covector $X \in \MMM$ with $F \not\subseteq \Flat(X)$, i.e. any covector $X$ which does not contribute to $\MMM^F$.
    Since $X$ is a circuit of $\MMM^F$, it easily follows that 
    \begin{equation}
        \prod_{b \in B} z_b \cdot \prod_{i \in \Supp(X)} y_i^{X(i)} \in \II(\Zpoints_\MMM) \quad \Rightarrow \quad 
        \prod_{b \in B} z_b \cdot \prod_{i \in \Supp(X)} y_i^{X(i)} \in \gr \, \II(\Zpoints_\MMM)
    \end{equation}
    where the implication holds because the polynomial on the left is homogeneous.

    (5) Let $F \in \LLL(\MMM)$ be a flat, let $B \subseteq F$ be a basic set for $F$, and let $X: (\III -F) \to \{+,-,0\}$ be a symmetric circuit of $\MMM^F$ with support size $s = \#\Supp(X) > 1$. Fix a nonempty and proper subset $\varnothing \subsetneq \JJJ \subsetneq \Supp(X)$ and let $\{\tilde{y}_i \,:\, i \in \Supp(X) \}$ be defined using $\JJJ$ as in Lemma~\ref{lem:two-values-attained-big}. Our goal is to establish the membership
    \begin{equation}
        \label{eqn:membership-four-one}
        \prod_{b \in B} z_b \cdot e_{s-1}(\tilde{y}_i \,:\, i \in \Supp(X)) \in \gr \, \II(\Zpoints_\MMM).
    \end{equation}
    Indeed, let $t$ be a new variable and consider the rational function
    \begin{equation}
        \label{eqn:membership-four-two}
        \prod_{b \in B} z_b \cdot \left(
            \frac{\prod_{i \in \Supp(X)} (1 + \tilde{y}_i  \cdot t)}{1+t}
        \right).
    \end{equation}
    Equation~\eqref{eqn:face-element-evaluation} implies that \eqref{eqn:membership-four-two} evaluates to 0 on any locus point $\zzz_X$ corresponding to a covector $X \in \MMM$ for which $F \not\subseteq \Flat(X)$. If $F \subseteq \Flat(X)$ so that $X$ contributes to $\MMM^F$, Lemma~\ref{lem:two-values-attained-big} applied to $\MMM^F$ implies that \eqref{eqn:membership-four-two} evaluates on $\zzz_X$ to a polynomial in $t$ of degree $<s-1$. Overall, we see that \eqref{eqn:membership-four-two} evaluates to a polynomial in $t$ of degree $< s-1$ on any point of the big locus $\Zpoints_\MMM$.  Taking the coefficient of $t^{s-1}$ shows
    \begin{equation}
        \prod_{b \in B} z_b \cdot (e_{s-1}(\tilde{y}_i \,:\, i \in \Supp(X)) - e_{s-2}(\tilde{y}_i \,:\, i \in \Supp(X)) + \cdots + (-1)^{s-1}) \in \II(\Zpoints_\MMM)
    \end{equation}
    and taking the top degree component gives the desired membership \eqref{eqn:membership-four-one}.
\end{proof}

Lemma~\ref{lem:big-ideal-containment} gives rise to the sequence of surjections $\widehat{S} \twoheadrightarrow \widehat{S}/\widetilde{I}_\MMM \twoheadrightarrow \widehat{S}/\widehat{I}_\MMM \twoheadrightarrow \RRR(\Zpoints_\MMM)$. We aim to show that $\widehat{I}_\MMM = \gr \, \II(\Zpoints_\MMM)$ so that this last surjection $\widehat{S}/\widehat{I}_\MMM \twoheadrightarrow \RRR(\Zpoints_\MMM)$ is an isomorphism. To do this, we use Lemma~\ref{lem:z-divisibility} to build a filtration on $\widehat{\VVV}_\MMM$ as follows.

Recall that $\LLL(\MMM)$ is partially ordered by containment. Since $\widetilde{I}_\MMM \subseteq \widehat{I}_\MMM$, Lemma~\ref{lem:z-divisibility} implies 
\begin{equation}
z_{F_2} \cdot (\widehat{S}/\widehat{I}_\MMM) \subseteq z_{F_1} \cdot (\widehat{S}/\widehat{I}_\MMM) \quad \text{whenever $F_1 \subseteq F_2.$}
\end{equation}
For any $F \in \LLL(\MMM)$, the assignment $(\widehat{S}/\widehat{I}_\MMM)_F := z_F \cdot \widehat{S}/\widehat{I}_\MMM$ is an $\LLL(\MMM)$-filtration on $\widehat{S}/\widehat{I}_\MMM$. (Since $\widehat{I}_\MMM$ contains the squares of all variables in $\widehat{S}$, the quotient $\widehat{S}/\widehat{I}_\MMM$ is a finite-dimensional vector space.)

We employ NBC sets to put our $\LLL(\MMM)$-filtration to good use. Let $<$ be an arbitrary but fixed total order on $\III$. The total order $<$ restricts to a total order on $\III - F$ for any flat $F \in \LLL(\MMM)$. Recall that $\NNN(\MMM^F)$ is the collection of NBC sets of the contraction $\MMM^F$; these are subsets of $\III-F$. We also fix a basic set $B(F)$ for each flat $F \in \LLL(\MMM)$. 

For any $F \in \LLL(\MMM)$, let $\NNN_F$ be the collection of monomials
\begin{equation}
    \NNN_F := \left\{ 
            \prod_{i \in N} y_i^+ \,:\, N \in \NNN(\MMM^F)
    \right\} \subseteq \widehat{S}.
\end{equation}  
Let $\widehat{\NNN} \subseteq \widehat{S}$ be the set of monomials
\begin{equation}
    \widehat{\NNN} := \bigsqcup_{F \in \LLL(\MMM)} \left(  \prod_{b \in B(F)} z_b \cdot \NNN_F \right)
\end{equation}
where the use of disjoint union $\sqcup$ is justified because distinct flats $F$ have distinct basic sets $B(F)$.

Our main structural result on $\RRR(\Zpoints_\MMM)$ requires one final definition. If $F \in \LLL(\MMM)$ is a flat of $\MMM$, Lemma~\ref{lem:basic-lemma} (1) states that every basic set $\BBB$ for $F$ has the same cardinality. We write $\codim(F)$ for this common cardinality $\# \BBB$.

\begin{theorem}
    \label{thm:big-locus-identification}
    Let $\MMM$ be a conditional oriented matroid on $\III$. We have the equality of ideals \[\widehat{I}_\MMM = \gr \, \II(\Zpoints_\MMM)\] in the polynomial ring $\widehat{S}$ so that $\RRR(\Zpoints_\MMM) = \widehat{S}/\widehat{I}_\MMM$. The set $\widehat{\NNN}$ descends to a basis of $\RRR(\Zpoints_\MMM)$. We have an isomorphism of graded $\FF$-vector spaces
    \[
    \RRR(\Zpoints_\MMM)  \cong \bigoplus_{F \in \LLL(\MMM)} \RRR(\Ypoints_{\MMM^F})(- \codim(F))
    \]
    where $\RRR(\Ypoints_{\MMM^F})(-d)$ is the graded vector space $\RRR(\Ypoints_{\MMM^F})$ with degree shifted up by $d$.
\end{theorem}

\begin{proof}
    The algebra $\widehat{S}/\widehat{I}_\MMM$ has a $\LLL(\MMM)$-filtration structure with subquotients
    \[
    (\widehat{S}/\widehat{I}_\MMM)_{=F} = \frac{(\widehat{S}/\widehat{I}_\MMM)_F}{\sum_{F \subsetneq F'} (\widehat{S}/\widehat{I}_\MMM)_{F'}}.
    \]
    Lemma~\ref{lem:big-ideal-containment} and Lemma~\ref{lem:P-filtration} yield inequalities
    \begin{equation}
    \label{eq:inequality-chain-big}
        \# \MMM = \# \Zpoints_\MMM = \dim \RRR(\Zpoints_\MMM) = \dim \widehat{S}/\gr \, \II(\Zpoints_\MMM) \leq \dim \widehat{S}/\widehat{I}_\MMM \\ \leq \sum_{F \in \LLL(\MMM)} \dim (\widehat{S}/\widehat{I}_\MMM)_{= F}.
    \end{equation}
    On the other hand, Lemma~\ref{lem:tope-contraction-count} gives
    \begin{equation}
    \label{eq:tope-equality}
        \# \MMM = \sum_{F \in \LLL(\MMM)} \# \TTT(\MMM^F).
    \end{equation}
    Combining \eqref{eq:tope-equality} with  the result \eqref{eq:nbc-tope} of Dorpalen-Barry, Proudfoot, and Wang implies
    \begin{equation}
        \label{eq:nbc-equality}
        \# \widehat{\NNN} = \sum_{F \in \LLL(\MMM)} \# \NNN(\MMM^F) = \sum_{F \in \LLL(\MMM)} \# \TTT(\MMM^F) = \# \MMM.
    \end{equation}
    We establish the following claim.

    {\bf Claim:} {\em Let $F \in \LLL(\MMM)$ be a flat. The set $\prod_{b \in \BBB(F)} z_b \cdot \NNN_F$ of monomials in $\widehat{S}$ descends to a spanning set of $(\widehat{S}/\widehat{I}_\MMM)_{=F}$ as an $\FF$-vector space.}

    Working towards a proof of the claim, we establish some relations in $(\widehat{S}/\widehat{I}_\MMM)_{=F}$. Since $\widetilde{I}_\MMM \subseteq \widehat{I}_\MMM$ we have $\prod_{c \in C} z_c = 0$ in $\widehat{S}/\widehat{I}_\MMM$ for any nonbasic set $C \subseteq \III$. Thanks to Lemma~\ref{lem:z-divisibility}, the filtration defining $(\widehat{S}/\widehat{I}_\MMM)_{= F}$ gives rise to the relation
    \begin{equation}
    \label{eq:basic-filtration-relation}
        \prod_{b \in B(F)} z_b \cdot z_i = 0 \quad \text{in $(\widehat{S}/\widehat{I}_\MMM)_{=F}$ for any $i \in \III$.}
    \end{equation}
    If $X: (\III -F) \to \{+,-,0\}$ is a symmetric circuit of $\MMM^F$ of support size $s = \# \Supp(X) > 1$, the relation \eqref{eq:basic-filtration-relation} implies
    \begin{equation}
    \label{eq:filtration-relation-one}
        \prod_{b \in B(F)} z_b \cdot e_{s-1}(y_i^{X(i)} \,:\, i \in \Supp(X)) = 0 \quad \text{in $(\widehat{S}/\widehat{I}_\MMM)_{= F}$}.
    \end{equation}
     Similarly, the relation \eqref{eq:basic-filtration-relation} implies 
    \begin{equation}
    \label{eq:filtration-relation-two}
        \prod_{b \in B(F)} z_b \cdot (y_i^+ + y_i^-) = 0 \quad \text{in $(\widehat{S}/\widehat{I}_\MMM)_{=F}$ for any $i \in \III$.}
    \end{equation}
    The relations of $\widehat{S}/\widehat{I}_\MMM$ in Definition~\ref{def:I-ideal} (1) imply
    \begin{equation}
    \label{eq:filtration-relation-three}
        \prod_{b \in B(F)} z_b \cdot (y_i^+)^2 = \prod_{b \in B(F)} z_b \cdot (y_i^-)^2 = \prod_{b \in B(F)} z_b \cdot y_i^+ \cdot y_i^- = 0 \quad \text{in $(\widehat{S}/\widehat{I}_\MMM)_{= F}$ for any $i \in \III.$}
    \end{equation}
    If $X: (\III -F) \to \{+,-,0\}$ is a circuit of $\MMM^F$, Definition~\ref{def:I-ideal} (4) implies
    \begin{equation}
        \label{eq:filtration-relation-four}
        \prod_{b \in B(F)} z_b \cdot \prod_{i \in \Supp(X)} y_i^{X(i)} = 0 \quad \text{in $(\widehat{S}/\widehat{I}_\MMM)_{= F}$.}
    \end{equation}
    Definition~\ref{def:I-ideal} (3) yields
    \begin{equation}
        \label{eq:filtration-relation-five}
        \prod_{b \in B(F)} z_b \cdot y_i^+ = 0 \quad \text{in $(\widehat{S}/\widehat{I}_\MMM)_{= F}$ for all $i \in F$}.
    \end{equation}
    
    With the above relations in hand, the claim is proven by a  straightening argument. The relations \eqref{eq:basic-filtration-relation}, \eqref{eq:filtration-relation-two}, and \eqref{eq:filtration-relation-five} imply that $(\widehat{\VVV}_\MMM)_{=F}$ is spanned by elements of the form
    \begin{equation}
    \label{eq:purported-contradiction}
        \prod_{b \in \BBB(F)} z_b \cdot m
    \end{equation}
    where $m$ is a monomial in $\{y_i^+ \,:\, i \in \III-F \}$. Let $\prec$ be the lexicographical term order on the monomials in $\FF[y_i^+ \,:\, i \in \III-F ]$ induced by the total order $<$ on $\III$. Working towards a contradiction, let $m$ be the $\prec$-minimal monomial in $\FF[y_i^+ \,:\, i \in \III-F ]$ such that \eqref{eq:purported-contradiction} does not lie in the span of $\prod_{b \in \BBB(F)} z_b \cdot \NNN_F$ when regarded as an element of $(R_\MMM)_{=F}$. The relations \eqref{eq:filtration-relation-two}, \eqref{eq:filtration-relation-three}, and \eqref{eq:filtration-relation-four} imply that $m$ is a squarefree monomial and that 
    \[
    \prod_{i \in \Supp(X)} y_i^+ \nmid m \quad \text{for any circuit $X : (\III -F) \to \{+,-,0\}$ of $\MMM^F$.}
    \]
    If $m \notin \prod_{b \in \BBB(F)} z_b \cdot \NNN_F$, there must be a symmetric circuit $X: (\III -F) \to \{+,-,0\}$ of $\MMM^F$ such that 
    \[
    \prod_{i \in \Supp(X)^\circ} y_i^+  \mid m.
    \]
    where $\Supp(X)^\circ = \Supp(X) - \{\min_<(\Supp(X))\}.$ The relation \eqref{eq:filtration-relation-one} implies 
    \[
    m = \text{ a $\ZZ$-linear combination of monomials $m'$ in $\{y_i^+ \,:\, i \in \III-F\}$ with $m' \prec m$} 
    \]
    in $(\widehat{S}/\widehat{I}_\MMM)_{=F}$, which contradicts the $\prec$-minimality of $m$ and proves the claim.

    We use the claim (and its proof) to derive the theorem as follows. The claim implies  
    \begin{equation}
    \label{eq:filtration-piece-inequality}
        \dim (\widehat{S}/\widehat{I}_\MMM)_{= F} \leq \# \NNN(\MMM^F) \quad \text{for all $F \in \LLL(\MMM).$}
    \end{equation}
    Combining \eqref{eq:inequality-chain-big}, \eqref{eq:tope-equality}, and \eqref{eq:filtration-piece-inequality} forces
    \begin{equation}
    \label{eq:filtration-piece-equality}
        \dim (\widehat{S}/\widehat{I}_\MMM)_{= F} = \# \NNN(\MMM^F) \quad \text{for all $F \in \LLL(\MMM).$}
    \end{equation}
    The inequalities in \eqref{eq:inequality-chain-big} are therefore equalities and we have
    \begin{equation}
        \widehat{I}_\MMM = \gr \, \II(\Zpoints_\MMM) \quad \text{so that } \quad 
        \RRR(\Zpoints_\MMM) = \widehat{S}/\widehat{I}_\MMM.
    \end{equation}
    The $\LLL(\MMM)$-filtration on $\widehat{S}/\widehat{I}_\MMM$ is therefore an $\LLL(\MMM)$-filtration on $\RRR(\Zpoints_\MMM)$, so the notation $\RRR(\Zpoints_\MMM)_{= F}$ for $F \in \LLL(\MMM)$ makes sense.
    
    The claim combined with \eqref{eq:filtration-piece-equality} imply that $\widehat{\NNN}$ descends to an $\FF$-basis of $\RRR(\Zpoints_\MMM)$. In particular (or by Lemma~\ref{lem:P-filtration}), we have an isomorphism of graded vector spaces
    \begin{equation}
    \label{eq:filtration-isomorphism}
        \RRR(\Zpoints_\MMM) \cong \bigoplus_{F \in \LLL(\MMM)} \RRR(\Zpoints_\MMM)_{=F}.
    \end{equation}

    Let $F \in \LLL(\MMM)$ be a flat. The theorem reduces to proving \[\RRR(\Zpoints_\MMM)_{=F} \cong \RRR(\Ypoints_{\MMM^F})(-\codim(F))\] as graded vector spaces; we do this as follows. Write $S_F$ for the polynomial ring 
    \[
    S_F := \FF[y_i^+, y_i^- \,:\, i \in \III - F]
    \]
    so that we have the ideal $\gr \, \II(\Ypoints_{\MMM^F}) \subseteq S_{F}$ and $S_F/\gr \, \II(\Ypoints_{\MMM^F}) =  \RRR(\Ypoints_{\MMM^F})$. Let $\widetilde{\varphi}_F$ be the composition
    \begin{equation}
        \widetilde{\varphi}_F: S_F \hookrightarrow \widehat{S} \twoheadrightarrow \RRR(\Zpoints_\MMM) \xrightarrow{ \, \, \times z_F \, \, } \RRR(\Zpoints_\MMM)_{\geq F} \twoheadrightarrow \RRR(\Zpoints_\MMM)_{= F}
    \end{equation}
    where the inclusion $S_F \hookrightarrow \widehat{S}$ comes from containment of variable sets. The map $\widetilde{\varphi}_F$ is a homomorphism of $S_F$-modules which shifts degree up by $\codim(F)$. Theorem~\ref{thm:small-ideal-generators} gives generators for the ideal $\gr \, \II(\Ypoints_{\MMM^F}) \subseteq S_F$. Comparing these genenerators with the relations \eqref{eq:filtration-relation-one}, \eqref{eq:filtration-relation-two}, \eqref{eq:filtration-relation-three}, and \eqref{eq:filtration-relation-four} gives $\gr \, \II(\Ypoints_{\MMM^F})\subseteq \Ker(\widetilde{\varphi}_F)$ so that $\widetilde{\varphi}_F$ induces an $S_F$-module homomorphism
    \begin{equation}
        \varphi_F : \RRR(\Ypoints_{\MMM^F}) \longrightarrow \RRR(\Zpoints_\MMM)_{=F}
    \end{equation}
    which shifts degree up by $\codim(F)$. The claim implies that $\varphi_F$ is surjective. Since 
    \[
    \dim \RRR(\Ypoints_{\MMM^F}) = \# \TTT(\MMM^F) = \# \NNN(\MMM^F) = \dim \RRR(\Zpoints_\MMM)_{=F},
    \]
    the map $\varphi_F$ is an isomorphism. 
\end{proof}

Theorem~\ref{thm:big-locus-identification} was phrased in terms of orbit harmonics rings; its VG ring avatar is as follows. Lemma~\ref{lem:tope-contraction-count} implies that the big and small VG rings satisfy
\begin{equation}
    \widehat{VG}_\MMM \cong \bigoplus_{F \in \LLL(M)} \VG_{\MMM^F}.
\end{equation}
Propositions~\ref{prop:small-locus-identification} and \ref{prop:big-orbit-harmonics-interpretations} give an equivalent formulation of the graded vector space isomorphism of Theorem~\ref{thm:big-locus-identification} in terms of big and small graded VG rings:
\begin{equation}
    \widehat{\VVV}_\MMM \cong \bigoplus_{F \in \LLL(\MMM)} \VVV_{\MMM^F} (-\codim(F)).
\end{equation}
This isomorphism yields a Hilbert series identity.

\begin{corollary}
    \label{cor:big-hilbert-series}
    For any conditional oriented matroid $\MMM$ on the ground set $\III$, one has the Hilbert series
    \[
    \Hilb(\widehat{\VVV}_\MMM;q) = \Hilb(\RRR(\Zpoints_\MMM);q) = \sum_{F \in \LLL(\MMM)} q^{\codim(F)} \cdot \Hilb(\VVV_{\MMM^F};q). 
    \]
\end{corollary}

Much of our main result Theorem~\ref{thm:big-locus-identification} holds over any commutative ground ring $A$.  Let $\widehat{S}^\ZZ$ be the polynomial ring
\[
\widehat{S}^A := A[y_i^+,y_i^-,z_i \,:\, i \in \III]
\]
and define $\widehat{I}^A_\MMM \subseteq \widehat{S}^A$ to be the ideal with the same generators as $\widehat{I}_\MMM$. By definition, the {\em big Varchenko--Gelfand ring with coefficients in A} is the quotient
\begin{equation}
    \widehat{\VVV}^A_\MMM := \widehat{S}^A/\widehat{I}^A_\MMM.
\end{equation}
Then $\widehat{\VVV}^A_\MMM$ is a graded $A$-algebra.

\begin{proposition}
    \label{prop:other-rings}
    The quotient $\widehat{\VVV}^A_\MMM$ is a free $A$-module with $A$-basis $\widehat{\NNN}$.
\end{proposition}

\begin{proof} (Sketch)
We start with the case $A = \ZZ$.
The argument proving Theorem~\ref{thm:big-locus-identification} (with $\LLL(\MMM)$-filtrations defined over $\ZZ$) goes through to show that the set $\widehat{\NNN}$ spans $\widehat{S}^\ZZ / \widehat{I}^\MMM_\ZZ$ over $\ZZ$; the key point is that the relations \eqref{eq:filtration-relation-one}-\eqref{eq:filtration-relation-five} appearing in the proof of Theorem~\ref{thm:big-locus-identification} involve only the coefficients $\pm 1$. Since any $\ZZ$-linear relation on $\widehat{\NNN}$ modulo $\widehat{I}^\ZZ_\MMM$ would induce a $\QQ$-linear relation on $\widehat{\NNN}$ modulo $\widehat{I}^\QQ_\MMM$, we see that $\widehat{S}^\ZZ/\widehat{I}^\ZZ_\MMM$ is a free $\ZZ$-module with $\ZZ$-basis $\widehat{\NNN}$. Applying the functor $A \otimes_\ZZ (-)$ gives the proposition. 
\end{proof}

\subsection{Equivariant structure} When the conditional oriented matroid $\MMM$ admits automorphisms, Theorem~\ref{thm:big-locus-identification} and Corollary~\ref{cor:big-hilbert-series} have natural equivariant enhancements. The relevant algebra is as follows.

Let $\symm_{\pm \III}$ be the group of signed permutations of the ground set $\III$. The group $\symm_{\pm \III}$ acts naturally on signed subsets $X$ of $\III$. An element $w \in \symm_{\pm \III}$ is an {\em automorphism} of $\MMM$ if $w \cdot X \in \MMM$ for all covectors $X \in \MMM$. We write $\Aut(\MMM)$ for the group of automorphisms of $\MMM$. The subset $\TTT(\MMM) \subseteq \MMM$ of topes is stable under the action of $\Aut(\MMM)$.

The group $\symm_{\pm \III}$ of signed permutations acts naturally on the vector space $\FF^{\III \times \{+,-\}}$ of dimension $2 \cdot \# \III$ and the larger vector space $\FF^{\III \times \{+,-,0\}}$ of dimension $3 \cdot \# \III$. Here changing the sign at $i \in \III$ interchanges the  basis vectors indexed by $(i,+)$ and $(i,-)$ while fixing the one indexed by $(i,0)$. The subgroup $\Aut(\MMM)$ of $\symm_{\pm \III}$ acts on the small locus $\Ypoints_\MMM \subseteq \FF^{\III \times \{+,-\}}$ and the big locus $\Zpoints_\MMM \subseteq \FF^{\III \times \{+,-,0\}}$. It is not hard to see that 
\[
\Ypoints_\MMM \cong_{\Aut(\MMM)} \TTT(\MMM) \quad \text{and} \quad 
\Zpoints_\MMM \cong_{\Aut(\MMM)} \MMM
\]
as $\Aut(\MMM)$-sets.

The action of $\symm_{\pm \III}$ on $\FF^{\III \times \{+,-\}}$ and $\FF^{\III \times \{+,-,0\}}$ induces a variable permuting action of $\symm_{\pm \III}$ on the coordinate rings 
\[
S = \FF[y_i^+,y_i^- \,:\, i \in \III] \quad \text{and} \quad \widehat{S} = \FF[y_i^+,y_i^-,z_i \,:\, i \in \III]
\]
of these vector spaces. The subgroup $\Aut(\MMM) \subseteq \symm_{\pm \III}$ acts on $S$ and $\widehat{S}$ by restriction. The group $\Aut(\MMM)$ acts on the quotient rings $\RRR(\Ypoints_\MMM) \cong \VVV_\MMM$ and $\RRR(\Zpoints_\MMM) \cong \widehat{\VVV}_\MMM.$ Orbit harmonics immediately gives the ungraded module structure of these quotient rings.

\begin{theorem}
    \label{thm:ungraded-module-identification}
    The ideals $I_\MMM \subseteq S$ and $\widehat{I}_\MMM \subseteq \widehat{S}$ are stable under the action of $\Aut(\MMM)$. If $\# \Aut(\MMM) \in \FF^\times$ we have isomorphisms of ungraded $\Aut(\MMM)$-modules 
    \begin{equation}
        \VVV_\MMM \cong \RRR(\Ypoints_\MMM) \cong  \FF[\TTT(\MMM)] \quad \text{and} \quad 
        \widehat{\VVV}_\MMM \cong \RRR(\Zpoints_\MMM) \cong  \FF[\MMM].
    \end{equation}
\end{theorem}

The graded $\Aut(\MMM)$-module structure of $\widehat{\VVV}_\MMM \cong \RRR(\Zpoints_\MMM)$ may be derived using the $\LLL(\MMM)$-filtration and Lemma~\ref{lem:P-filtration-equivariant}. We set up the relevant algebra.

The automorphism group $\Aut(\MMM)$ acts naturally on the poset $\LLL(\MMM)$. If $F \in \LLL(\MMM)$ is a flat, we write $\Aut(\MMM)_F \subseteq \Aut(\MMM)$ for the stabilizer subgroup
\[
\Aut(\MMM)_F := \{ w \in \Aut(\MMM) \,:\, w \cdot F = F \}.
\]
Let $\LLL(\MMM)/\Aut(\MMM)$ be the family of $\Aut(\MMM)$-orbits $[F]$ in $\LLL(\MMM)$.

As in the proof of Theorem~\ref{thm:big-locus-identification}, the ring $\RRR(\Zpoints_\MMM) \cong \widehat{\VVV}_\MMM$ admits an $\LLL(\MMM)$-filtration given by \[\RRR(\Zpoints_\MMM)_F := z_F \cdot \RRR(\Zpoints_\MMM).\]  If $F \in \LLL(\MMM)$ is a flat, $B \subseteq F$ is a basic set of $F$, and $w \in \Aut(\MMM)$ is an automorphism of $\MMM$, it not hard to see that $w \cdot B \subseteq w \cdot F$ is a basic set for $w \cdot F$. It follows that $w \cdot z_F = z_{w \cdot F}$ in $\RRR(\Zpoints_\MMM)$ and our $\LLL(\MMM)$-filtration is $\Aut(\MMM)$-equivariant.

For any $F \in \LLL(\MMM)$, the stabilizer $\Aut(\MMM)_F$ acts on $\RRR(\Ypoints_{\MMM^F})$ as follows. The disjoint union decomposition $\III = F \sqcup (\III -F)$ gives rise to a subgroup $\symm_{\pm F} \times \symm_{\pm (\III -F)} \subseteq \symm_{\pm \III}$. The signed permutation group $\symm_{\pm (\III -F)}$ permutes the variables of $S_F = \FF[y_i^+,y_i^- \,:\, i \in \III -F]$ as before. We extend this to an action of the product group $\symm_{\pm F} \times \symm_{\pm (\III -F)}$ by letting $\symm_{\pm F}$ act trivially. It is easily seen that $\Aut(\MMM)_F \subseteq \symm_{\pm F} \times \symm_{\pm (\III - F)}$, so that $\Aut(\MMM)_F$ acts on $S_F$ by restriction.
Since any element $w \in \Aut(\MMM)_F$ restricts to $\III -F$ to give an automorphism of $\MMM^F$, the ideal $\gr \, \II(\Ypoints_{\MMM^F}) \subseteq S_F$ is stable under the action of $\Aut(\MMM)_F$. This gives $\RRR(\Ypoints_{\MMM^F}) = S_F/\gr \, \II(\Ypoints_{\MMM^F})$ the structure of a graded $\Aut(\MMM)_F$-module.

\begin{theorem}
    \label{thm:graded-module-structure}
    Let $\MMM$ be a conditional oriented matroid on the ground set $\III$. Assume that $\# \Aut(\MMM) \in \FF^\times$. We have the isomorphism of graded $\Aut(\MMM)$-modules
    \[
    \RRR(\Zpoints_\MMM) \cong \bigoplus_{[F] \in \LLL(\MMM)/\Aut(\MMM)} \Ind_{\Aut(\MMM)_F}^{\Aut(\MMM)}\RRR(\Ypoints_{\MMM^F})(-\codim(F))
    \]
    where $\Ind_{\Aut(\MMM)_F}^{\Aut(\MMM)}(-)$ is induction of group representations and $(-\codim(F))$ shifts degree up by $\codim(F)$.
\end{theorem}

In terms of big and small graded VG rings, Theorem~\ref{thm:graded-module-structure} reads
\begin{equation}
    \widehat{\VVV}_\MMM \cong_{\Aut(\MMM)} \bigoplus_{[F] \in \LLL(\MMM)/\Aut(\MMM)} \VVV_{\MMM^F}(-\codim(F))
\end{equation}
as graded $\Aut(\MMM)$-modules.

\begin{proof}
Lemma~\ref{lem:P-filtration-equivariant} and Theorem~\ref{thm:big-locus-identification} give the isomorphism of graded $\Aut(\MMM)$-modules
\begin{equation}
    \RRR(\Zpoints_\MMM) \cong \bigoplus_{[F] \in \LLL(\MMM)/\Aut(\MMM)} \Ind_{\Aut(\MMM)_F}^{\Aut(\MMM)} \RRR(\Zpoints_\MMM)_{=F}.
\end{equation}
To complete the proof, we observe that for $F \in \LLL(\MMM)$, the isomorphism
\begin{equation}
    \varphi_F: \RRR(\Ypoints_{\MMM^F}) \xrightarrow{\, \, \sim \, \, } \RRR(\Zpoints_\MMM)_{=F}
\end{equation}
of vector spaces from the proof of Theorem~\ref{thm:big-locus-identification} commutes with the action of $\Aut(\MMM)_F$. This holds because $\varphi_F$ is defined using the inclusion of variables $y_i^\pm \mapsto y_i^\pm$ for $i \in \III - F$.
\end{proof}

\section{The braid arrangement: An example}
\label{sec:Braid}

This section illustrates the results in this paper for the important special case where $\MMM = \MMM_n$ is the (conditional) oriented matroid arising from the braid arrangement
\[
\AAA_n = \{ x_i - x_j = 0 \,:\, 1 \leq  i < j \leq n \}
\]
in $\RR^n$ where the convex set $\KKK$ is the full space $\RR^n$. We have the ground set $\III = {[n] \choose 2}$. Covectors in $\MMM_n$ (i.e. faces of $\AAA_n$) are indexed by ordered set partitions of $[n]$. Topes in $\TTT(\MMM_n)$ are indexed by permutations in $\symm_n$. The action of $\symm_n$ on $\RR^n$ by coordinate permutation induces an action of $\symm_n$ on $\MMM_n$ by automorphisms.

\subsection{Hilbert series}  The small locus $\Ypoints_{\MMM_n}$ lives in $\FF^{{[n] \choose 2} \times \{+,-\}}$ and has points indexed by chambers in $\AAA_n$, i.e. permutations in $\symm_n$. For $n=3$, these points have coordinates as in the following table where we write permutations $w \in \symm_3$ in one-line notation.

\[
\begin{tabular}{c | c c c c c c }
$w \in \symm_3$ & $12,+$ & $12,-$ & $13,+$ & $13,-$ & $23,+$ & $23,-$ \\ \hline
123 & 1 & 0 & 1 & 0 & 1 & 0 \\
213 & 0 & 1 & 1 & 0 & 1 & 0 \\
132 & 1 & 0 & 1 & 0 & 0 & 1 \\
231 & 0 & 1 & 0 & 1 & 1 & 0 \\
312 & 1 & 0 & 0 & 1 & 0 & 1 \\
321 & 0 & 1 & 0 & 1 & 0 & 1
\end{tabular}
\]

Proposition~\ref{prop:small-locus-identification} identifies $\RRR(\Ypoints_{\MMM_n}) \cong \VVV_{\MMM_n}$ with the (small) graded VG ring associated to the braid arrangement $\AAA_n$. It follows that
\begin{equation}
\label{eq:small-braid-hilbert}
    \Hilb( \RRR(\Ypoints_{\MMM_n});q) = \Hilb( \VVV_{\MMM_n};q) = q \cdot (q+1) \cdot (q+2) \cdots (q+n-1) 
\end{equation}
More generally, if $\AAA$ is a central free hyperplane arrangement with exponents $e_1,\dots,e_n$ and $\MMM$ is the associated oriented matroid, one has (see e.g. \cite{OT}) 
\begin{equation}
    \Hilb(\RRR(\Ypoints_\MMM);q) = \Hilb(\VVV_\MMM;q) = \prod_{i=1}^n (q + e_i).
\end{equation}
Equation~\ref{eq:small-braid-hilbert} is the {\em Stirling distribution} on the symmetric group $\symm_n$. It is the generating function for the statistic $\cyc: \symm_n \to \ZZ_{\geq 0}$ where $\cyc(w)$ is the number of cycles of $w$.

The big locus $\Zpoints_{\MMM_n}$ lives in $\FF^{{[n] \choose 2} \times \{+,-,0\}}$ and has points indexed by ordered set partitions of $[n]$. For $n=3$, the coordinates of these points are as follows.

\[
\begin{tabular}{c | c c c c c c c c c}
    & $12,+$ & $12,-$ & $12,0$ & $13,+$ & $13,-$ & $13,0$ & $23,+$ & $23,-$ & $23,0$ \\ \hline 
$(1 \mid 2 \mid 3)$ & 1 & 0 & 0 & 1 & 0 & 0 & 1 & 0 & 0 \\
$(2 \mid 1 \mid 3)$ & 0 & 1 & 0 & 1 & 0 & 0 & 1 & 0 & 0 \\
$(1 \mid 3 \mid 2)$ & 1 & 0 & 0 & 1 & 0 & 0 & 0 & 1 & 0 \\
$(2 \mid 3 \mid 1)$ & 0 & 1 & 0 & 0 & 1 & 0 & 1 & 0 & 0 \\
$(3 \mid 1 \mid 2)$ & 1 & 0 & 0 & 0 & 1 & 0 & 0 & 1 & 0 \\
$(3 \mid 2 \mid 1)$ & 0 & 1 & 0 & 0 & 1 & 0 & 0 & 1 & 0 \\
$(12 \mid 3)$ & 0 & 0 & 1 & 1 & 0 & 0 & 1 & 0 & 0 \\
$(13 \mid 2)$ & 1 & 0 & 0 & 0 & 0 & 1 & 0 & 1 & 0 \\
$(23 \mid 1)$ & 0 & 1 & 0 & 0 & 1 & 0 & 0 & 0 & 1 \\
$(1 \mid 23)$ & 1 & 0 & 0 & 1 & 0 & 0 & 0 & 0 & 1 \\
$(2 \mid 13)$ & 0 & 1 & 0 & 0 & 0 & 1 & 1 & 0 & 0 \\
$(3 \mid 12)$ & 0 & 0 & 1 & 0  & 1 & 0 & 0 & 1 & 0 \\
$(123)$ & 0 & 0 & 1 & 0 & 0 & 1 & 0 & 0 & 1
\end{tabular}
\]

The poset of flats $\LLL(\MMM_n)$ is the poset $\Pi_n$ of (unordered) set partitions of $[n]$ where $\pi \leq \sigma$ if $\pi$ refines $\sigma$. Corollary~\ref{cor:big-hilbert-series} implies
\begin{equation}
    \label{eq:big-hilbert-series-braid}
    \Hilb(\RRR(\Zpoints_{\MMM_n});q) = \Hilb(\widehat{\VVV}_{\MMM_n};q) = \sum_{\pi \in \Pi_n} q^{n - \# \pi } \cdot   (q+1) \cdot (2q+1) \cdots ((\# \pi - 1)q + 1)
\end{equation}
where $\# \pi$ is the number of blocks of the set partition $\pi$. The first few of these Hilbert series are given in the following table.
\[
\begin{tabular}{c | c}
$n$ & $\Hilb(\RRR(\Zpoints_{\MMM_n});q) = \Hilb(\widehat{\VVV}_{\MMM_n};q) $\\ \hline 
1 & 1 \\
2 & $2q + 1$ \\
3 & $6q^2 + 6q + 1$ \\ 
4 & $26q^3 + 36q^2 + 12q + 1$ \\
5 & $150q^4 + 250q^3 + 120q^2 + 20q + 1$
\end{tabular}
\]
Equation~\eqref{eq:big-hilbert-series-braid} has the following combinatorial interpretation.
\begin{proposition}
    \label{prop:necklaces}
    For any $n$ we have 
    \begin{equation}
    \label{eq:necklaces}
    \sum_{\pi \in \Pi_n} q^{n - \# \pi } \cdot   (q+1)  \cdot (2q+1) \cdots ((\# \pi - 1)q+1) =  
    \sum_{k=0}^n a_{n,k} \cdot q^k
\end{equation}
where $a_{n,k}$ counts unordered $(n-k)$-tuples of necklaces with beads labeled by subsets $S \subseteq [n]$ such that the bead labels form a set partition of $[n]$. 
\end{proposition}

\begin{proof}
    Recall that we have the generating function identity
    \begin{equation}
    \label{eq:stirling-number-first-kind-identity}
        \sum_{w \in \symm_n} q^{n-\cyc(w)} = (q+1)(2q+1) \cdots ((n-1)q+1).
    \end{equation}
    Given a set partition $\pi \in \Pi_n$, let $\symm_\pi \cong \symm_{\# \pi}$ be the group of permutations of the blocks of $\pi$. For $w \in \symm_\pi$, write $\cyc(w)$ for the number of cycles of $w$. For example, if $n  =8$ and $$\pi = \{17 \mid 2 5 \mid 3 \mid 4 \mid  68 \}$$ then one possible element  $w \in \symm_\pi$ is given in cycle notation as 
    $$
    w = (17, \, 68, \, 4)(25, \, 3).
    $$
    In this case we have $\cyc(w) = 2$. Restating the identity \eqref{eq:stirling-number-first-kind-identity} for $\symm_\pi$ gives 
    \begin{equation}
        \label{eq:stirling-variant}
        \sum_{w \in \symm_\pi} q^{\# \pi - \cyc(w)} = (q+1)(2q+1) \cdots ((\# \pi-1)q+1).
    \end{equation}
    Applying \eqref{eq:stirling-variant}, the summand of Equation~\eqref{eq:necklaces} indexed by $\pi$ may be expressed as 
    \begin{equation}
        q^{n - \# \pi } \cdot   (q+1)  \cdot (2q+1) \cdots ((\# \pi - 1)q+1) = \sum_{w \in \symm_\pi} q^{n-\cyc(w)}.
    \end{equation}
    Summing over all $\pi \in \Pi_n$ gives the proposition.
\end{proof}

\begin{example}
Let $n = 3$. The following family of objects witnesses  $a_{3,2} = 6$.    
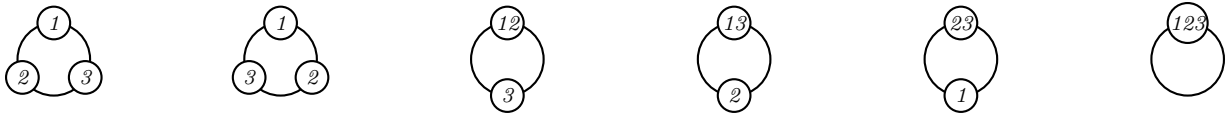
\begin{scriptsize}
\begin{center}
    \begin{tikzpicture}[scale=1.2]

  % ---------- First necklace ----------
\begin{scope}
  \draw[thick] (0,0) circle (0.4);

  % clockwise: 1,2,3
  \foreach \ang/\lab in {90/1,210/2,330/3} {
    \filldraw[fill=white, thick] (\ang:0.4) circle (0.18);
    \node at (\ang:0.4) {\lab};
  }
\end{scope}

% ---------- Second necklace (shifted right) ----------
\begin{scope}[xshift=2.5cm]
  \draw[thick] (0,0) circle (0.4);

  % clockwise: 1,3,2
  \foreach \ang/\lab in {90/1,210/3,330/2} {
    \filldraw[fill=white, thick] (\ang:0.4) circle (0.18);
    \node at (\ang:0.4) {\lab};
  }
\end{scope}

% ---------- Third necklace (shifted right) ----------
\begin{scope}[xshift=5cm]
  \draw[thick] (0,0) circle (0.4);

  % clockwise: 1,3,2
  \foreach \ang/\lab in {90/12,270/3} {
    \filldraw[fill=white, thick] (\ang:0.4) circle (0.18);
    \node at (\ang:0.4) {\lab};
  }
\end{scope}

\begin{scope}[xshift=7.5cm]
  \draw[thick] (0,0) circle (0.4);

  % clockwise: 1,3,2
  \foreach \ang/\lab in {90/13,270/2} {
    \filldraw[fill=white, thick] (\ang:0.4) circle (0.18);
    \node at (\ang:0.4) {\lab};

  }
\end{scope}

\begin{scope}[xshift=10cm]
  \draw[thick] (0,0) circle (0.4);

  % clockwise: 1,3,2
  \foreach \ang/\lab in {90/23,270/1} {
    \filldraw[fill=white, thick] (\ang:0.4) circle (0.18);
    \node at (\ang:0.4) {\lab};

  }
\end{scope}

\begin{scope}[xshift=12.5cm]
  \draw[thick] (0,0) circle (0.4);

  % Single bead
  \filldraw[fill=white, thick] (90:0.4) circle (0.22);
  \node at (90:0.4) {123};

\end{scope}

\end{tikzpicture}
\end{center}
\end{scriptsize}
The following family of objects witnesses $a_{3,1}=6$.
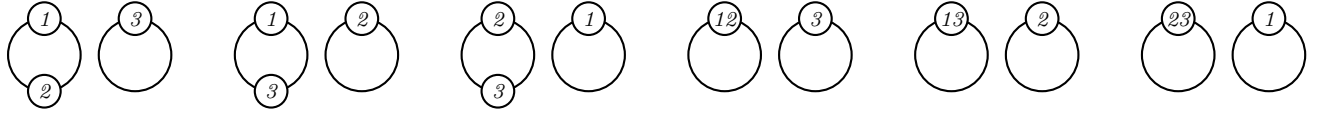
\begin{scriptsize}
    \begin{center}
        \begin{tikzpicture}[scale=1.2]
      
\begin{scope}[xshift=0cm]
  \draw[thick] (0,0) circle (0.4);

  % clockwise: 1,3,2
  \foreach \ang/\lab in {90/1,270/2} {
    \filldraw[fill=white, thick] (\ang:0.4) circle (0.18);
    \node at (\ang:0.4) {\lab};

  }
\end{scope}
\begin{scope}[xshift=1cm]
  \draw[thick] (0,0) circle (0.4);

  % clockwise: 1,3,2
  \foreach \ang/\lab in {90/3} {
    \filldraw[fill=white, thick] (\ang:0.4) circle (0.18);
    \node at (\ang:0.4) {\lab};

  }
\end{scope}

\begin{scope}[xshift=2.5cm]
  \draw[thick] (0,0) circle (0.4);

  % clockwise: 1,3,2
  \foreach \ang/\lab in {90/1,270/3} {
    \filldraw[fill=white, thick] (\ang:0.4) circle (0.18);
    \node at (\ang:0.4) {\lab};

  }
\end{scope}
\begin{scope}[xshift=3.5cm]
  \draw[thick] (0,0) circle (0.4);

  % clockwise: 1,3,2
  \foreach \ang/\lab in {90/2} {
    \filldraw[fill=white, thick] (\ang:0.4) circle (0.18);
    \node at (\ang:0.4) {\lab};

  }
\end{scope}
\begin{scope}[xshift=5cm]
  \draw[thick] (0,0) circle (0.4);

  % clockwise: 1,3,2
  \foreach \ang/\lab in {90/2,270/3} {
    \filldraw[fill=white, thick] (\ang:0.4) circle (0.18);
    \node at (\ang:0.4) {\lab};

  }
\end{scope}
\begin{scope}[xshift=6cm]
  \draw[thick] (0,0) circle (0.4);

  % clockwise: 1,3,2
  \foreach \ang/\lab in {90/1} {
    \filldraw[fill=white, thick] (\ang:0.4) circle (0.18);
    \node at (\ang:0.4) {\lab};

  }
\end{scope}
\begin{scope}[xshift=7.5cm]
  \draw[thick] (0,0) circle (0.4);

  % clockwise: 1,3,2
  \foreach \ang/\lab in {90/12} {
    \filldraw[fill=white, thick] (\ang:0.4) circle (0.18);
    \node at (\ang:0.4) {\lab};

  }
\end{scope}
\begin{scope}[xshift=8.5cm]
  \draw[thick] (0,0) circle (0.4);

  % clockwise: 1,3,2
  \foreach \ang/\lab in {90/3} {
    \filldraw[fill=white, thick] (\ang:0.4) circle (0.18);
    \node at (\ang:0.4) {\lab};

  }
\end{scope}
\begin{scope}[xshift=10cm]
  \draw[thick] (0,0) circle (0.4);

  % clockwise: 1,3,2
  \foreach \ang/\lab in {90/13} {
    \filldraw[fill=white, thick] (\ang:0.4) circle (0.18);
    \node at (\ang:0.4) {\lab};

  }
\end{scope}
\begin{scope}[xshift=11cm]
  \draw[thick] (0,0) circle (0.4);

  % clockwise: 1,3,2
  \foreach \ang/\lab in {90/2} {
    \filldraw[fill=white, thick] (\ang:0.4) circle (0.18);
    \node at (\ang:0.4) {\lab};

  }
\end{scope}
\begin{scope}[xshift=12.5cm]
  \draw[thick] (0,0) circle (0.4);

  % clockwise: 1,3,2
  \foreach \ang/\lab in {90/23} {
    \filldraw[fill=white, thick] (\ang:0.4) circle (0.18);
    \node at (\ang:0.4) {\lab};

  }
\end{scope}
\begin{scope}[xshift=13.5cm]
  \draw[thick] (0,0) circle (0.4);

  % clockwise: 1,3,2
  \foreach \ang/\lab in {90/1} {
    \filldraw[fill=white, thick] (\ang:0.4) circle (0.18);
    \node at (\ang:0.4) {\lab};

  }
\end{scope}
        \end{tikzpicture}
    \end{center}
\end{scriptsize}
Finally, the equality $a_{3,0}=1$ is witnessed as follows.
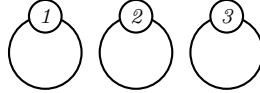
\begin{scriptsize}
    \begin{center}
        \begin{tikzpicture}[scale=1.2]
\begin{scope}[xshift=0cm]
  \draw[thick] (0,0) circle (0.4);

  % clockwise: 1,3,2
  \foreach \ang/\lab in {90/1} {
    \filldraw[fill=white, thick] (\ang:0.4) circle (0.18);
    \node at (\ang:0.4) {\lab};

  }
\end{scope}

\begin{scope}[xshift=1cm]
  \draw[thick] (0,0) circle (0.4);

  % clockwise: 1,3,2
  \foreach \ang/\lab in {90/2} {
    \filldraw[fill=white, thick] (\ang:0.4) circle (0.18);
    \node at (\ang:0.4) {\lab};

  }
  \end{scope}
  \begin{scope}[xshift=2cm]
  \draw[thick] (0,0) circle (0.4);

  % clockwise: 1,3,2
  \foreach \ang/\lab in {90/3} {
    \filldraw[fill=white, thick] (\ang:0.4) circle (0.18);
    \node at (\ang:0.4) {\lab};

  }
  \end{scope}

        \end{tikzpicture}
    \end{center}
\end{scriptsize}
\end{example}

\subsection{$\symm_n$-module structure}
Assume $\FF$ has characteristic 0 or $\mathrm{char}(\FF) > n$.
Let $\CCC_n := \Conf_n(\RR^3)$ be the configuration space of ordered $n$-tuples $(v_1,\dots,v_n)$ of pairwise distinct points in $\RR^3$. The cohomology ring $H^*(\CCC_n)$ is supported in even degrees, and is therefore commutative.

The configuration space $\CCC_n$ admits a natural action of the symmetric group $\symm_n$, and this action is inherited by its cohomology ring $H^*(\CCC_n)$. Moseley proved \cite{Moseley} that $H^*(\CCC_n) \cong \VVV_{\MMM_n}$ as graded $\symm_n$-algebras. It follows that
\begin{equation}
\label{eq:small-cohomology-identification}
\RRR(\Ypoints_{\MMM_n}) \cong \VVV_{\MMM_n} \cong H^*(\CCC_n)
\end{equation}
as graded $\symm_n$-modules. The representations \eqref{eq:small-cohomology-identification} give a graded deformation of the regular representation of $\symm_n$.  Brauner proved \cite[Thm. 4.10]{Brauner} that the graded pieces of \eqref{eq:small-cohomology-identification} carry the Eulerian representations of $\symm_n$ (see also \cite{ABR}). Brauner's result holds in the greater generality of coincidental real reflection groups.

The orbits of $\LLL(\MMM_n)/\Aut(\MMM_n) \cong \Pi_n/\symm_n$ are indexed by integer partitions $\lambda \vdash n$. For an integer partition $\lambda = (\lambda_1,\dots,\lambda_k) \vdash n$, let $G(\lambda) \subseteq \symm_n$ be the subgroup which stabilizes the set partition
\[
\{ \{1, 2, \dots, \lambda_1\}, \, \{\lambda_1+1, \lambda_1 + 2, \dots, \lambda_1 + \lambda_2\}, \, \dots , \, \{n-\lambda_k+1, \dots, n-1,n\} \}.
\]
If $m(i)$ is the multiplicity of $i$ as a part of $\lambda$, one has a group isomorphism
\[
G(\lambda) \cong \prod_{i \geq 1} \symm_{m(i)} \wr \symm_{i}
\]
between $G(\lambda)$ and a direct product of wreath products.

For an integer partition $\lambda \vdash n$, let $\CCC_\lambda$ be the space of ordered $n$-tuples $(v_1,\dots,v_n)$ of points in $\RR^3$ such that
$\{ w \in \symm_n \,:\, w \cdot (v_1,\dots,v_n) = (v_1,\dots,v_n) \}$ is the parabolic subgroup $\symm_\lambda \subseteq \symm_n$. For example, if $\lambda = (3,2,2) \vdash 7$ we have $\CCC_\lambda = \{ (v,v,v,w,w,u,u) \,:\, v,w,u \in \RR^3 \text{ distinct} \}$.
One has a homeomorphism
\begin{equation}
    \CCC_\lambda \cong \CCC_{\ell(\lambda)} 
\end{equation}
where $\ell(\lambda)$ is the number of parts of $\lambda$. The group $G(\lambda)$ acts on $\CCC_\lambda$ by coordinate permutation. The cohomology ring $H^*(\CCC_\lambda)$ inherits a graded action of $G(\lambda)$.
Combining Theorem~\ref{thm:graded-module-structure} and Moseley's result \eqref{eq:small-cohomology-identification} one has the identity of graded $\symm_n$-characters
\begin{equation}
\label{eq:big-braid-character}
    \ch_q (\RRR(\Zpoints_{\MMM_n})) = \ch_q(\widehat{\VVV}_{\MMM_n}) = \sum_{\lambda \vdash n} q^{n-\ell(\lambda)} \ch_q \left( \Ind_{G(\lambda)}^{\symm_n} H^*(\CCC_\lambda)  \right).
\end{equation}
The character \eqref{eq:big-braid-character} is a graded deformation of the permutation action of $\symm_n$ on ordered set partitions of $[n]$.

\subsection{Topological interpretation?} A missing piece in our story is a topological interpretation of the big VG algebra $\widehat{\VVV}_{\MMM_n}$. We leave finding such an interpretation as an open problem.

\begin{problem}
    \label{prob:topological-interpretation}
    Is there a natural enlargement $\CCC_n \subseteq \widehat{\CCC}_n$ of the configuration space $\CCC_n$ which carries an action of $\symm_n$ so that
    \[
    \RRR(\Zpoints_{\MMM_n}) \cong \widehat{\VVV}_{\MMM_n} \cong H^*(\widehat{\CCC}_n)
    \]
    as graded $\symm_n$-algebras?
\end{problem}

The cohomology of a space $\widehat{\CCC}_n$ solving Problem~\ref{prob:topological-interpretation} would be concentrated in even degrees with sum of Betti numbers given by the number of ordered set partitions of $[n]$. When $n = 2$, the hypothetical space $\widehat{\CCC}_2$ would have cohomology presentation
\[
H^*(\widehat{\CCC}_2) = \frac{\FF[y^+,y^-,z]}{I + (y^+ + y^- + z)}
\]
where $I$ is generated by degree 2 monomials in $\{y^+,y^-,z\}$. The space $\widehat{\CCC}_2$ would therefore have
 Betti numbers $(1,0,2,0,0,\dots)$. 

In light of the work of Proudfoot \cite{Proudfoot} and Moseley \cite{Moseley}, Problem~\ref{prob:topological-interpretation} may be best solved with a space $\widehat{\CCC}_n$ equipped with an action of the 1-dimensional torus $T = \CC^*$. Let $\widehat{\Zpoints}_{\MMM_n} \subseteq \FF^{{[n] \choose 2} \times \{+,-,0\}} \times \FF$ be the variety obtained by placing the locus $\Zpoints_{\MMM_n}$ at height 1 and joining every point in this locus to the origin with a line. The work of \cite{Proudfoot, Moseley, DPW, CMR} inspires hope that the $T$-equivariant cohomology of $\widehat{\CCC}_n$ admits presentation
$H^*_T(\widehat{\CCC}_n) = (\widehat{S} \otimes_\FF \FF[t])/\II(\widehat{\Zpoints}_{\MMM_n}).$ The work of Dorpalen-Barry, Proudfoot, and Wang \cite[Thm. 1.1]{DPW} motivates an extension of Problem~\ref{prob:topological-interpretation} to arbitrary real arrangements $\AAA$ in $V$ and arbitrary open convex sets $\KKK \subseteq V.$ Remark~\ref{rmk:HW} suggests that a solution to Problem~\ref{prob:topological-interpretation} might involve an enlargement of matroid Schubert varieties.

\end{document}